\documentclass{amsart}
\usepackage{amssymb,mathtools}
\usepackage[british]{babel}
\usepackage{enumitem}
\usepackage[latin1]{inputenc}
\usepackage{url}

\usepackage{bussproofs}
\newenvironment{bprooftree}
  {\leavevmode\hbox\bgroup}
  {\DisplayProof\egroup}

\newtheorem{theorem}{Theorem}
\numberwithin{theorem}{section}
\newtheorem{corollary}[theorem]{Corollary}
\newtheorem{lemma}[theorem]{Lemma}
\newtheorem{proposition}[theorem]{Proposition}
\theoremstyle{definition}
\newtheorem{definition}[theorem]{Definition}
\newtheorem{remark}[theorem]{Remark}

\title{A mathematical commitment without computational strength}
\author{Anton Freund}

\begin{document}

\begin{abstract}
We present a new manifestation of G\"odel's second incompleteness theorem and discuss its foundational significance, in particular with respect to Hilbert's program. Specifically, we consider a proper extension of Peano arithmetic ($\mathbf{PA}$) by a mathematically meaningful axiom scheme that consists of $\Sigma^0_2$-sentences. These sentences assert that each computably enumerable (\mbox{$\Sigma^0_1$-definable} without parameters) property of finite binary trees has a finite basis. Since this fact entails the existence of polynomial time algorithms, it is important for computer science. On a technical level, our axiom scheme is a variant of an independence result due to Harvey Friedman. At the same time, the meta-mathematical properties of our axiom scheme distinguish it from most known independence results: Due to its logical complexity, our axiom scheme does not add computational strength. The only known method to establish its independence relies on G\"odel's second incompleteness theorem. In contrast, G\"odel's theorem is not needed for typical examples of $\Pi^0_2$-independence (such as the Paris-Harrington principle), since computational strength provides an extensional invariant on the level of $\Pi^0_2$-sentences.
\end{abstract}


\maketitle

\section{Summary of mathematical results}

This paper consists of mathematical results and a foundational discussion. The former are summarized in the present section; the latter can be found in Section~\ref{sect:foundational}. In the remaining sections we provide detailed proofs of all mathematical claims.

First and foremost, our paper is based on a result by Dick de Jongh (unpublished; cf.~the introduction to~\cite{schmidt75}) and Diana~Schmidt~\cite{schmidt-habil}: The embeddability relation on finite binary trees yields a well partial order with maximal order type~$\varepsilon_0$ (see below for an explanation). Harvey~Friedman~\cite{simpson85} has show that this type of result yields statements of finite combinatorics that are independent of important mathematical axiom systems. Against this background, many arguments in the present paper may be considered folklore. Nevertheless we find it worthwhile to give an explicit presentation, not least because the arguments are rather sensitive with respect to quantifier complexity and the presence of parameters. At some places we provide more details than the expert may find necessary. The aim is to make the paper as accessible and self-contained as possible.

We write $\mathcal B$ for the set of finite binary trees. More precisely, we assume that each tree has a distinguished root node, that nodes have either zero or two children, and that left and right child can be distinguished. Furthermore, we identify isomorphic trees. Formally, we view~$\mathcal B$ as the least fixed point of the following inductive clauses:
\begin{enumerate}[label=(\roman*)]
\item There is an element $\circ\in\mathcal B$ (the tree that consists of a single root node).
\item Given $s$ and $t$ in $\mathcal B$, we obtain an element $\circ(s,t)\in\mathcal B$ (the tree in which the root has left subtree~$s$ and right subtree $t$).
\end{enumerate}
For $s,t\in\mathcal B$ we write $s\leq_{\mathcal B}t$ if there is a tree embedding of $s$ into~$t$. Such an embedding can either map the root to the root and the immediate subtrees of $s$ into the corresponding subtrees of~$t$; or it maps all of $s$ into one subtree of~$t$. Hence we have $\circ\leq_{\mathcal B}t$ for any $t\in\mathcal B$; we have $s\leq_{\mathcal B}\circ$ precisely for $s=\circ$; and we have
\begin{equation*}
\circ(s_0,s_1)\leq_{\mathcal B}\circ(t_0,t_1)\quad\Leftrightarrow\quad\begin{cases}
s_0\leq_{\mathcal B}t_0\text{ and }s_1\leq_{\mathcal B}t_1,\\
\text{or }\circ(s_0,s_1)\leq_{\mathcal B}t_i\text{ for some }i\in\{0,1\}.
\end{cases}
\end{equation*}
These clauses provide a recursive definition of~$\leq_{\mathcal B}$.

Recall that a partial order consists of a set $X$ and a binary relation~$\leq_X$ on~$X$ that is reflexive, antisymmetric and transitive. A finite or infinite sequence $x_0,x_1,\ldots$ in~$X$ is called good if there are indices $i<j$ such that we have $x_i\leq_X x_j$; otherwise, the sequence is called bad. If there is no infinite bad sequence, then $(X,\leq_X)$ is called a well partial order~(wpo). Equivalently, a partial order $(X,\leq_X)$ is a wpo if, and only if, every subset $Y\subseteq X$ has a finite ``basis" $a\subseteq Y$ with the following property: for any $y\in Y$ there is an $x\in a$ with $x\leq_X y$ (cf.~the argument in Remark~\ref{rmk:KSigma-true} below).

If $X$ is a wpo, then all its linearizations are well orders (since a strictly decreasing sequence in a linearization would be a bad sequence in~$X$). Hence the order type of each linearization is an ordinal number. The supremum of these ordinals is called the maximal order type of~$X$. As shown by D.~de~Jongh and R.~Parikh~\cite{deJongh-Parikh}, the maximal order type of any wpo is realized by one of its linearizations (i.\,e.~the supremum is a maximum).

Kruskal's theorem~\cite{kruskal60} implies that $(\mathcal B,\leq_{\mathcal B})$ is a well partial order. We point out that the theorem applies to arbitrary (i.\,e.~not necessarily binary)  finite trees; the ``most general" version of Kruskal's theorem is investigated in~\cite{frw-kruskal}. Concerning the binary case, de Jongh and Schmidt have proved the finer result that $\mathcal B$ has maximal order type~$\varepsilon_0$, which is the least fixed point of ordinal exponentiation with base~$\omega$ (read~\cite[Theorem~II.2]{schmidt-habil} in combination with the example after~\cite[Definition~I.15]{schmidt-habil}). A classical result of G.~Gentzen~\cite{gentzen36,gentzen43} establishes~$\varepsilon_0$ as the proof theoretic ordinal of Peano arithmetic~($\mathbf{PA}$). This explains the connection with independence results.

In the present paper we consider the binary Kruskal theorem in the context of first order arithmetic; an introdution to this setting can be found in~\cite{hajek91}. We will be particularly interested in questions of quantifier complexity: Recall that a formula lies in the class $\Delta^0_0=\Sigma^0_0=\Pi^0_0$ if it does only contain bounded quantifiers. Since the latter range over a finite domain, the truth of closed $\Delta^0_0$-formulas is uniformly decidable. A $\Sigma^0_{n+1}$-formula ($\Pi^0_{n+1}$-formula) has the form $\exists_x\varphi$ (the form $\forall_x\varphi$), where $\varphi$ is a $\Pi^0_n$-formula ($\Sigma^0_n$-formula). Recall that the $\Sigma^0_1$-formulas correspond to the computably enumerable relations. A relation is $\Delta^0_1$-definable (in~$\mathbf{PA}$) if it has a $\Sigma^0_1$-definition and a $\Pi^0_1$-definition (which $\mathbf{PA}$ proves to be equivalent). The $\Delta^0_1$-relations coincide with the decidable ones.

Working in $\mathbf{PA}$, the elements of $\mathcal B$ can be represented by numerical codes for finite sets of sequences with entries from~$\{0,1\}$. Note that the relations $s\in\mathcal B$ and $s\leq_{\mathcal B}t$ are $\Delta^0_1$-definable in~$\mathbf{PA}$. As mentioned above, the fact that $\mathcal B$ is a wpo can be expressed in terms of a finite basis property. To state the latter we abbreviate
\begin{equation*}
\exists^{\operatorname{fin}}_a\psi(a):\equiv\exists_a(\text{``$a\in\mathbb N$ codes a finite set"}\land\psi(a)).
\end{equation*}
In the context of~$\mathbf{PA}$ it is natural to focus on definable sets. Given a formula $\varphi(s)$ with a distinguished free variable, the finite basis property for $\{s\in\mathcal B\,|\,\varphi(s)\}\subseteq\mathcal B$ can be formalized as
\begin{equation*}
\mathcal K\varphi:\equiv\exists^{\operatorname{fin}}_{a\subseteq\mathcal B}(\forall_{s\in a}\varphi(s)\land\forall_{t\in\mathcal B}(\varphi(t)\rightarrow\exists_{s\in a}s\leq_{\mathcal B}t)).
\end{equation*}
Note that the quantifiers with subscript $s\in a$ are bounded, since $a$ is a code for a finite set (cf.~\cite[Lemma~I.1.32]{hajek91}); in contrast, the quantifiers with subscripts $a\subseteq\mathcal B$ and $t\in\mathcal B$ are unbounded. The symbol $\mathcal K$ alludes to Kruskal's theorem, which implies that all instances $\mathcal K\varphi$ are true (see~Remark~\ref{rmk:KSigma-true} for details). We will be most interested in the axiom scheme
\begin{equation*}
\mathcal K\Sigma^-_1:=\{\mathcal K\varphi\,|\,\text{``$\varphi$ a $\Sigma^0_1$-formula with exactly one free variable"}\}.
\end{equation*}
The superscript of $\Sigma^-_1$ emphasizes the fact that no further free variables are allowed. This ensures that each instance of $\mathcal K\Sigma^-_1$ is a closed $\Sigma^0_2$-formula.

To motivate the restrictions on the quantifier complexity and the parameters, we recall the notion of computational strength: A computable function $f:\mathbb N\to\mathbb N$ is provably total in a suitable theory $\mathbf T$ if the latter proves $\forall_x\exists!_y\varphi(x,y)$ for some \mbox{$\Sigma^0_1$-definition}~$\varphi$ of the graph of~$f$ (where $\exists!$ abbreviates the existence of a unique witness). The computational strength of a theory is commonly identified with the collection of its provably total computable functions.

It is known that the computational strength of a theory does not increase when we add a true $\Pi^0_1$-sentence $\psi$ as an axiom. Essentially, this is due to the fact that the $\Sigma^0_1$-formula $\psi\to\varphi(x,y)$ defines the same graph as $\varphi(x,y)$ (note that the definition of provably total function is extensional). A simple but fundamental observation shows that the same is true for closed $\Sigma^0_2$-axioms: It suffices to note that any true $\Sigma^0_2$-sentence $\exists_x\psi(x)$ follows from some true $\Pi^0_1$-instance $\psi(\overline n)$ (see Proposition~\ref{prop:computational-strength} for details). Note that we may not be able to compute the correct witness~$n\in\mathbb N$; this issue will resurface at the end of the present section.

The general facts from the previous paragraph imply that $\mathbf{PA}+\mathcal K\Sigma^-_1$ has the same computational strength as~$\mathbf{PA}$. At this point it is crucial that we exclude parameters: If the $\Sigma^0_1$-formula $\varphi$ contains further free variables, then the universal closure of~$\mathcal K\varphi$ is a $\Pi^0_3$-formula, so that our argument does no longer apply. Note that the version with parameters can be expressed by a single $\Pi^0_3$-sentence (rather than a scheme), due to the existence of a universal computably enumerable set.

Next, we explain why $\mathbf{PA}+\mathcal K\Sigma^-_1$ is a proper extension of~$\mathbf{PA}$. Based on a notation system for the ordinal~$\varepsilon_0$ (see Section~\ref{sect:basis-to-ti} for details), transfinite induction can be expressed in first order arithmetic: Given a formula $\psi(\alpha)$ with a distinguished free variable, we set
\begin{equation*}
\mathcal{TI}(\varepsilon_0,\psi):\equiv \forall_{\gamma\prec\varepsilon_0}(\forall_{\beta\prec\gamma}\psi(\beta)\rightarrow\psi(\gamma))\rightarrow\forall_{\alpha\prec\varepsilon_0}\psi(\alpha).
\end{equation*}
The scheme of parameter-free $\Pi^0_1$-induction up to~$\varepsilon_0$ is the collection
\begin{equation*}
\mathcal{TI}(\varepsilon_0,\Pi^-_1):=\{\mathcal{TI}(\varepsilon_0,\psi)\,|\,\text{``$\psi$ a $\Pi^0_1$-formula with exactly one free variable"}\}.
\end{equation*}
In Section~\ref{sect:basis-to-ti} we show that each instance of $\mathcal{TI}(\varepsilon_0,\Pi^-_1)$ can be proved in $\mathbf{PA}+\mathcal K\Sigma^-_1$. This is a straightforward consequence of the fact that $\varepsilon_0$ is bounded by (and in fact equal to) the maximal order type of $\mathcal B$. Nevertheless we find it worthwhile to give a detailed proof, which pays attention to the quantifier complexities and the role of parameters. Gentzen~\cite{gentzen36} has used $\Pi^0_1$-induction up to~$\varepsilon_0$ to establish the consistency of~$\mathbf{PA}$. This induction does not require parameters, as we will check in Section~\ref{sect:ti-to-rfn}. Hence the consistency of $\mathbf{PA}$ can be proved in $\mathbf{PA}+\mathcal K\Sigma^-_1$. The latter must thus be a proper extension, due to G\"odel's second incompleteness theorem.

In Section~\ref{sect:reification} we review the proof that $\mathcal B$ has maximal order type $\varepsilon_0$. Based on this fact, we can also show that each instance of $\mathcal K\Sigma^-_1$ is provable in $\mathbf{PA}+\mathcal{TI}(\varepsilon_0,\Pi^-_1)$. To complete the picture, we relate transfinite induction and reflection. Let $\operatorname{Pr}_{\mathbf{PA}}(\varphi)$ be a standard formalization of the statement that the formula with code~$\varphi$ is provable in~$\mathbf{PA}$ (see~\cite[Section~I.4(a)]{hajek91}; we will also write $\varphi$ for $\overline{\ulcorner\varphi\urcorner}$). Given a sentence~$\varphi$ of first order arithmetic, we put
\begin{equation*}
\operatorname{Rfn}_{\mathbf{PA}}(\varphi):\equiv\operatorname{Pr}_{\mathbf{PA}}(\varphi)\rightarrow\varphi.
\end{equation*}
The local (i.\,e.~parameter-free) $\Sigma^0_2$-reflection principle over $\mathbf{PA}$ is the collection
\begin{equation*}
\operatorname{Rfn}_{\mathbf{PA}}(\Sigma^0_2):\equiv\{\operatorname{Rfn}_{\mathbf{PA}}(\varphi)\,|\,\text{``$\varphi$ a closed $\Sigma^0_2$-formula"}\}.
\end{equation*}
Due to G.~Kreisel and A.~L\'evy~\cite{kreisel68}, uniform reflection over~$\mathbf{PA}$ is equivalent to $\varepsilon_0$-induction for formulas with parameters. We will show that the proof can be adapted to the parameter free case. This results in Theorem~\ref{thm:basis-ind-rfl}, which asserts
\begin{equation*}
\mathbf{PA}+\mathcal K\Sigma^-_1\,\equiv\,\mathbf{PA}+\mathcal{TI}(\varepsilon_0,\Pi^-_1)\,\equiv\,\mathbf{PA}+\operatorname{Rfn}_{\mathbf{PA}}(\Sigma^0_2).
\end{equation*}
In view of Goryachev's theorem, we can conclude the following (see Corollary~\ref{cor:goryachev-basis}): Over Peano arithmetic, the $\Pi^0_1$-consequences of $\mathcal K\Sigma^-_1$ are precisely those of the finitely iterated consistency statements for~$\mathbf{PA}$. Due to another result of Kreisel and L\'evy~\cite{kreisel68}, we can also deduce that $\mathbf{PA}+\mathcal K\Sigma^-_1$ is not contained in any consistent extension of~$\mathbf{PA}$ by a computably enumerable set of~$\Pi^0_2$-sentences (see Corollary~\ref{cor:Sigma2-essential}).

\subsubsection*{Acknowledgements.} I am very grateful to Lev Beklemishev for our inspiring discussions and his helpful comments on a first version of this paper.

\section{Foundational considerations}\label{sect:foundational}

In the previous section we have presented an extension of Peano arithmetic by an axiom scheme $\mathcal K\Sigma^-_1$ that is related to Kruskal's theorem. The present section is concerned with the foundational significance of this extension.

Let us first recall some aspects of Hilbert's program; for a more thorough discussion and further references we refer to the introduction by R.~Zach~\cite{zach19}. To secure the abstract methods that are central to modern mathematics, Hilbert wanted to justify them by finitist reasoning about natural numbers, which he views as ``extralogical concrete objects that are intuitively present as immediate experience prior to all throught"~\cite[p.~171]{hilbert-unendliche}. (All quotations from~\cite{hilbert-unendliche,hilbert-grundlagen} are translated as in~\cite{heijenoort}.) The status of the natural numbers entails that certain statements about them are finitistically meaningful. This includes, first of all, statements which assert that a given tuple of numbers satisfies some primitive recursive relation. Such a statement can be verified explicitly, which explains its priviledged role, but also entails---as Hilbert~\cite[p.~165]{hilbert-unendliche} puts it---that it is ``of no essential interest when considered by itself". In addition, one admits universal statements with verifiable instances. According to Hilbert~\cite[p.~173]{hilbert-unendliche}, such a statement can be accepted as ``a hypothetical judgement that comes to assert something when a numeral is given". In contrast, unbounded existential statements are not seen as finitistically meaningful, as ``one cannot [\dots] try out all numbers"~\cite[p.~73]{hilbert-grundlagen}. At the same time, Hilbert~\cite[p.~77f]{hilbert-grundlagen} emphasizes the fact that existential statements play an extremely fruitful role in abstract mathematics. One could even be tempted to say that abstract notions acquire meaning through their role in the mathematical development, a position that seems to resonate with the following statement by Hilbert~\cite[p.~79]{hilbert-grundlagen}:
\begin{quote}
``To make it a universal requirement that each individual formula [\dots] be interpretable by itself is by no means reasonable; on the contrary, a theory by its very nature is such that we do not need to fall back upon intuition or meaning in the midst of some argument."
\end{quote}
However, such a conception of meaning is very different from the finitist one.

The extent of finitist reasoning is commonly identified with primitive recursive arithmetic ($\mathbf{PRA}$). This identification has been justified by W.~Tait~\cite{tait81}; in~\cite{tait02} he lists and refutes some objections. A quantifier-free formulation seems to be most appropriate: In a such a setting, one can only express statements that are finistically meaningful; universal statements correspond to open formulas. To make our considerations as accessible as possible, we will, nevertheless, work in the usual framework of first order arithmetic with quantifiers. Following C.~Smorynski~\cite{smorynski-incompleteness}, we agree to identify the finitistically meaningful statements with the~$\Pi^0_1$-sentences.

More specifically, then, Hilbert's program suggested to formalize all of abstract mathematics as an axiom system~$\mathbf T$. In order to obtain a finitist justification, one was supposed to prove the consistency of $\mathbf T$ in the theory~$\mathbf{PRA}$. At this point it is important to note that consistency is not merely a minimal requirement: If the consistency of a theory $\mathbf T$ is provable in $\mathbf{PRA}$, then the latter proves all \mbox{$\Pi^0_1$-theorems} of $\mathbf T$, i.\,e.~all results that are finitistically meaningful (see~\cite[p.~78f]{hilbert-grundlagen}). G\"odel's incompleteness theorems show that Hilbert's program cannot be carried out: It is impossible for $\mathbf T$ to prove its own consistency; a fortiori, the consistency of $\mathbf T$ cannot be established in the weaker theory~$\mathbf{PRA}$.

Despite G\"odel's theorems, the aims of Hilbert's program have been achieved to an astonishing extent: A substantial part of contemporary mathematics can indeed be formalized in rather weak axiom systems (see e.\,g.~the work of S.~Feferman~\cite{feferman-scientifically-applicable}, as well as U.~Kohlenbach's proof mining program~\cite{kohlenbach-proof-mining}). In view of these positive results, it is all the more intriguing to ask: Are there natural mathematical theorems that can be expressed but not proved in~$\mathbf{PRA}$, or in some stronger theory? To count as a natural theorem, the unprovable statement should arise from mathematical practice; it should not involve the logical notions of proof or model. In particular, consistency statements (which are unprovable by G\"odel's theorem) are not seen as examples of this type.

We do have good examples of true $\Pi^0_2$-statements that are unprovable in relevant axiom systems: The Paris-Harrington principle cannot be proved in Peano arithmetic~\cite{paris-harrington}; Friedman's miniaturization of Kruskal's theorem is independent of an even stronger system~\cite{simpson85}, which is associated with predicative mathematics. The situation is less satisfactory when it comes to $\Pi^0_1$-sentences, which are most important from the finitist viewpoint: The independent statement due to S.~Shelah~\cite{shelah-pi1-independence} involves notions from model theory, so that its status as a natural mathematical theorem can be questioned. Friedman has presented work on \mbox{$\Pi^0_1$-}independence from Zermelo-Fraenkel set theory (see e.\,g.~\cite{friedman-pi01-status}), but his results are not yet published in final form. In our opinion, the search for mathematical $\Pi^0_1$-sentences that are independent of relevant axiom systems remains one of the most interesting challenges in mathematical logic.

The axiom scheme $\mathcal K\Sigma^-_1$ from the previous section does not settle the challenge of natural $\Pi^0_1$-independence. The latter can, nevertheless, serve as a benchmark that helps us to assess the foundational significance of $\mathcal K\Sigma^-_1$. In the rest of this section we carry out such an assessment.

First, we will argue that $\mathcal K\Sigma^-_1$ is a natural mathematical commitment. In the previous section we have seen that $\mathcal K\Sigma^-_1$ is a restricted version of Kruskal's theorem. The latter is firmly established as a natural result of mathematical practice. Hence it remains to argue that the restrictions that lead to $\mathcal K\Sigma^-_1$ are natural as well.

In formulating $\mathcal K\Sigma^-_1$, we have restricted Kruskal's theorem in two ways: Firstly, we have decided to work with binary rather than arbitrary finite trees. This restriction makes it easier to determine the precise strength of~$\mathcal K\Sigma^-_1$ (i.\,e.~to prove the equivalence with transfinite induction and local reflection), but it is not essential: If we extend our axiom scheme to arbitrary finite trees, then it will imply the consistency of stronger axiom systems; at the same time, it will still not increase the computational strength, since it also consists of~$\Sigma^0_2$-statements. The graph minor theorem of N.~Robertson and P.~Seymour~\cite{rob-sey-gm} suggests a very intriguing axiom scheme that is even stronger (cf.~\cite{friedman-robertson-seymour}) but does not have computational strength either (for the same general reason). In summary, the restriction to binary trees is purely pragmatic and does not change the general foundational behaviour. Secondly, the scheme $\mathcal K\Sigma^-_1$ is a restriction of Kruskal's theorem insofar as it demands a finite basis for computably enumerable---rather than arbitrary---sets of trees. In the following we give two justifications for the restriction to computably enumerable sets.

The first justification is that $\mathcal K\Sigma^-_1$ suffices for certain applications in computer science: Assume that $P$ is an upwards closed property of finite binary trees, which means that $P(s)$ and $s\leq_{\mathcal B}t$ imply $P(t)$. Often (but not always, cf.~\cite[Theorem~3]{fellows-langston-88}) one will already know that~$P$ is decidable. Then~$P$ can be defined by a $\Sigma^0_1$-formula, and $\mathcal K\Sigma^-_1$ yields a finite $a\subseteq\mathcal B$ such that $P(t)$ is equivalent to~$\exists_{s\in a}s\leq_{\mathcal B}t$. The latter can be decided in polynomial time (in the size of~$t$). The author knows of no concrete applications in the context of trees, but the analogous argument for the graph minor relation has many applications (see e.\,g.~\cite{fellows-langston-92}).

The second justification for the restriction to computably enumerable sets is based on the idea that one can have reasons to accept $\mathcal K\Sigma^-_1$ but not the full Kruskal theorem for binary trees. To make this plausible we recall that $\mathcal K\Sigma^-_1$ is equivalent to parameter-free $\Pi^0_1$-induction up to~$\varepsilon_0$. The latter is no stronger than induction for decidable (i.\,e.~finitistically meaningful) properties, still up to~$\varepsilon_0$ (see e.\,g.~\cite[Lemma~4.5]{sommer95}). From a finitist standpoint it makes sense to accept this induction principle but not the second order statement that $\varepsilon_0$ is well-founded, which would be required for the binary Kruskal theorem. Indeed, Tait~\cite[p.~411]{tait02} states that Kreisel~\cite{kreisel58} accepts quantifier-free induction up to each ordinal below~$\varepsilon_0$ as finitist. Also, G.~Takeuti's justification of transfinite induction is supposed to ``involve `Gedankenexperimente' [thought experiments] only on clearly defined operations applied to some concretely given figures"~\cite[p.~97]{takeuti-proof-theory}.

Next, we discuss the fact that $\mathcal K\Sigma^-_1$ is a scheme rather than a single statement. In the previous section we have explained that $\mathbf{PA}+\mathcal K\Sigma^-_1$ proves the consistency of~$\mathbf{PA}$. Of course, this proof involves only finitely many instances $\mathcal K\varphi_1,\dots,\mathcal K\varphi_n$. However, we see no basis for the claim that these particular instances constitute a natural mathematical commitment---in contrast to the axiom scheme as a whole. In this sense our reference to an axiom scheme is essential. What does this entail? We think that the answer depends on our attitude towards independence phenomena.

One possibility is to think of independent statements as ``unsolvable conjectures". More explicitly, one might imagine a mathematician immersed in Peano arithmetic, who is challenged to prove or refute the Paris-Harrington principle. The independence result tells us that this mathematician can never succeed. This conception of independence is clearly concerned with single statements rather than schemes. However, one can also think of independence in terms of ``potential axioms". For example, one may view the principle of induction for arbitrary first order formulas as a mathematical commitment beyond the finitist standpoint. This example shows that schemes play a natural role within such a conception of independence.

A broad conception of independence may even incorporate rules, in addition to axiom schemes. In the present context it is interesting to consider the rule
\begin{equation*}
\begin{bprooftree}
\AxiomC{$\forall_{\gamma\prec\varepsilon_0}(\forall_{\beta\prec\gamma}\psi(\beta)\rightarrow\psi(\gamma))$}
\UnaryInfC{$\forall_{\alpha\prec\varepsilon_0}\psi(\alpha)$}
\end{bprooftree}
\end{equation*}
of $\Pi^0_1$-induction along~$\varepsilon_0$, which allows us to infer $\forall_{\alpha\prec\varepsilon_0}\psi(\alpha)$ once we have given a proof of $\forall_{\gamma\prec\varepsilon_0}(\forall_{\beta\prec\gamma}\psi(\beta)\rightarrow\psi(\gamma))$, where $\psi(\alpha)$ can be any $\Pi^0_1$-formula without further free variables. Note that the rule does not commit us to the contrapositive of the corresponding axiom, i.\,e.~to the least element principle. Hence the rule avoids certain existential commitments, which is well motivated in a finitist context. As shown by L.~Beklemishev~\cite[Theorem~3]{beklemishev-provability-algebras}, the closure of $\mathbf{PA}$ under the rule of \mbox{$\Pi^0_1$-}induction along~$\varepsilon_0$ proves the same theorems as the extension of $\mathbf{PA}$ by finitely iterated consistency statements. Note that the rule does not refer to logical notions such as proof or model. Insofar as induction up to $\varepsilon_0$ is a result of mathematical practice, we have a mathematical commitment on the level of $\Pi^0_1$-statements.

Finally, we discuss the fact that $\mathcal K\Sigma^-_1$ consists of $\Sigma^0_2$-statements rather than $\Pi^0_1$-statements. At the end of the previous section we have mentioned that there is no computably enumerable set $\Psi$ of $\Pi^0_1$-sentences (or even $\Pi^0_2$-sentences) such that $\mathbf{PA}+\Psi$ is consistent and contains $\mathbf{PA}+\mathcal K\Sigma^-_1$. This shows that our use of $\Sigma^0_2$-sentences is essential in a rather strong sense.

As mentioned above, many of the known independence results for $\mathbf{PA}$ are concerned with $\Pi^0_2$-sentences. Extending Hilbert's view on $\Pi^0_1$-sentences, one could see $\Pi^0_2$-sentences as ``hypothetical judgement[s]"~\cite[p.~173]{hilbert-unendliche} of complexity~$\Sigma^0_1$. This might suggest that $\Pi^0_2$-sentences are less abstract---in the finitist sense---than \mbox{$\Sigma^0_2$-}statements. From this viewpoint, the independence of $\mathcal K\Sigma^-_1$ would be less significant than the known independence results. An argument that supports the significance of~$\Sigma^0_2$-independence will be given below. First, we give another explanation for the fact that $\Pi^0_2$-independence is more prominent in the existing literature.

Gentzen's ordinal analysis shows that each purported proof of a contradiction can be reduced to a proof with smaller ordinal label. To establish consistency, one can use this reduction in two different ways: In the present paper, we invoke induction on $\alpha\prec\varepsilon_0$ to show that no proof with label~$\alpha$ can produce a contradiction. This avoids parameters but involves a universal quantification over proofs with given ordinal label; it leads to local $\Sigma^0_2$-reflection, which has complexity $\Sigma^0_2$. Alternatively, Gentzen's reduction shows that  a purported proof~$p$ of a contradiction leads to a strictly decreasing sequence of ordinals, which is primitive recursive with parameter~$p$. One can then invoke the primitive recursive well-foundedness of~$\varepsilon_0$. This leads to uniform $\Sigma^0_1$-reflection (see~\cite[Theorem~4.5]{friedman95}), which is a $\Pi^0_2$-statement. It seems that the second approach is preferred in the finitist literature. For example, Takeuti writes that the consistency proof is based on the following~\cite[p.~92]{takeuti-proof-theory}:
\begin{quote}
``Whenever a concrete method of constructing decreasing sequences of ordinals is given, any such decreasing sequence must be finite."
\end{quote}
This preference may help to explain the pre-eminence of $\Pi^0_2$-independence. As an exception, we mention that L.~Beklemishev and A.~Visser~\cite{beklemishev-visser05} have characterized the $\Sigma^0_n$-consequences of $\mathbf{PA}$ (and of its fragments) in terms of iterated reflection. Kreisel~\cite{kreisel-herbrand-symposium} has initiated work on finiteness theorems of complexity~$\Sigma^0_2$, but here the focus is on proof-mining rather than independence.

The significance of $\Sigma^0_2$-independence is related to the notions of provably total function and computational strength, which we have recalled in the previous section. An independent $\Pi^0_2$-statement will typically add a provably total function: For the Paris-Harrington principle this is the case by~\cite[Theorem~3.2]{paris-harrington}; the general claim is plausible in view of \cite[Theorems~2.24 and~4.5]{friedman95} and \cite[Theorem~5]{friedman-proof-length}. In contrast, we have seen that $\mathcal K\Sigma^-_1$ does not increase the computational strength of~$\mathbf{PA}$. 

The fact that $\mathcal K\Sigma^-_1$ does not add provably total functions is interesting in its own right, but it becomes even more relevant in view of the following: The notion of computational strength is a relatively robust extensional invariant. Bounds on provably total functions can be established without the use of G\"odel's theorem, e.\,g.~by induction over cut-free infinite proofs (see~\cite{buchholz-wainer-87}). This means that G\"odel's theorem is not needed to prove that the Paris-Harrington principle is independent of $\mathbf{PA}$ (see \cite{cichon83} for an analogous argument with respect to Goodstein's theorem). It appears that no similar invariants are available on the level of~$\Sigma^0_2$-statements. The only known proof of the fact that $\mathbf{PA}$ does not prove all instances of $\mathcal K\Sigma^-_1$ appeals to G\"odel's theorem. In our opinion, this means that $\mathcal K\Sigma^-_1$ is a conceptually different and foundationally significant manifestation of mathematical independence.

\section{Analyzing the computational strength}

In this section we give a detailed proof of the claim that $\mathcal K\Sigma^-_1$ does not increase the computational strength of~$\mathbf{PA}$. As preparation, we need to show that all instances of $\mathcal K\Sigma^-_1$ are true. In the following remark we argue in a strong meta theory; this will later be superseded by a proof in $\mathbf{PA}+\mathcal{TI}(\varepsilon_0,\Pi^-_1)$ (see Proposition~\ref{prop:ti-to-basis}).

\begin{remark}\label{rmk:KSigma-true}
As a consequence of Kruskal's theorem~\cite{kruskal60}, the partial order $(\mathcal B,\leq_{\mathcal B})$ does not contain any infinite bad sequence. We will use this fact to justify an arbitrary instance
\begin{equation*}
\mathcal K\varphi\equiv\exists^{\operatorname{fin}}_{a\subseteq\mathcal B}(\forall_{s\in a}\varphi(s)\land\forall_{t\in\mathcal B}(\varphi(t)\rightarrow\exists_{s\in a}s\leq_{\mathcal B}t))
\end{equation*}
of the axiom scheme $\mathcal K\Sigma^-_1$. Aiming at a contradiction, assume that $\mathcal K\varphi$ is false. By a bad $\varphi$-sequence we mean a bad sequence $t_0,t_1,\ldots\subseteq\mathcal B$ such that $\varphi(t_i)$ holds for each~$i$. Note that the empty sequence is a bad $\varphi$-sequence. Furthermore, each bad $\varphi$-sequence $t_0,\dots,t_{n-1}$ can be extended into a bad \mbox{$\varphi$-sequence}~$t_0,\dots,t_{n-1},t_n$. To see that this is the case, consider $a:=\{t_0,\dots,t_{n-1}\}$. As $\forall_{s\in a}\varphi(s)$ holds, the assumption that $\mathcal K\varphi$ is false yields an element $t_n\in\mathcal B$ with $\varphi(t_n)$ and $\forall_{s\in a}s\not\leq_{\mathcal B}t_n$. The latter ensures that $t_0,\dots,t_{n-1},t_n$ is still bad. By dependent choice we now get an infinite bad $\varphi$-sequence, which contradicts Kruskal's theorem.
\end{remark}

As explained in the introduction, the following is due to the general fact that $\mathcal K\Sigma^-_1$ consists of true $\Sigma^0_2$-sentences. The argument is folklore, but we provide details in order to make the paper as accessible as possible.

\begin{proposition}\label{prop:computational-strength}
The provably total functions of $\mathbf{PA}+\mathcal K\Sigma^-_1$ and of~$\mathbf{PA}$ coincide.
\end{proposition}
\begin{proof}
Consider a provably total function~$f:\mathbb N\to\mathbb N$ of~$\mathbf{PA}+\mathcal K\Sigma^-_1$. For some $\Sigma^0_1$-definition $\theta(x,y)$ of the graph of $f$, there are $\Sigma^0_1$-formulas $\varphi_0,\dots,\varphi_{n-1}$ (each with a single free variable) such that we have
\begin{equation*}
\mathbf{PA}+\{\mathcal K\varphi_i\,|\,i<n\}\vdash\forall_x\exists!_y\theta(x,y).
\end{equation*}
To show that $f$ is a provably total function of $\mathbf{PA}$, we will define the graph of~$f$ by a modified $\Sigma^0_1$-formula $\theta'(x,y)$ such that $\mathbf{PA}$ alone proves $\forall_x\exists!_y\theta(x,y)'$. For this purpose we observe that the conjunction $\mathcal K\varphi_0\land\dots\land\mathcal K\varphi_{n-1}$ is equivalent to a true $\Sigma^0_2$-sentence~$\exists_m\psi(m)$. Pick a number~$n\in\mathbb N$ such that the $\Pi^0_1$-sentence $\psi(\overline n)$ is true. Then write
\begin{equation*}
\exists_z\theta_0(x,y,z)\equiv\psi(\overline n)\to\theta(x,y)
\end{equation*}
for a $\Delta^0_0$-formula~$\theta_0$. Since $\psi(\overline n)$ is true and implies each instance $\mathcal K\varphi_i$, we do have
\begin{gather*}
f(k)=m\quad\Leftrightarrow\quad\mathbb N\vDash\exists_z\theta_0(\overline k,\overline m,z),\\
\mathbf{PA}\vdash\forall_x\exists_y\exists_z\theta_0(x,y,z).
\end{gather*}
However, if $\mathbf{PA}$ does not prove $\psi(\overline n)$, then it will not prove that the value~$y$ is unique. It is well known that one can restore uniqueness by minimizing over the code of the pair~$\langle y,z\rangle$. Note that minimizing over~$y$ alone would lead out of the \mbox{$\Sigma^0_1$-formulas}: the minimal~$y$ that satisfies $\exists_z\theta_0(x,y,z)$ is specified by a $\Delta^0_2$-formula. To provide details we write $w=\langle y,z\rangle$ for a $\Delta^0_1$-definition of Cantor's pairing function; recall that $w=\langle y,z\rangle$ implies $y,z\leq w$. Let $\theta'(x,y)$ be the $\Sigma^0_1$-formula
\begin{equation*}
\exists_w(\exists_{z\leq w}(w=\langle y,z\rangle\land\theta_0(x,y,z))\land\forall_{w'<w}\forall_{y',z'\leq w'}(w'=\langle y',z'\rangle\rightarrow\neg\theta_0(x,y',z'))).
\end{equation*}
It is straightforward to see that $\theta'$ defines $f$ and that $\mathbf{PA}$ proves $\forall_x\exists!y\theta'(x,y)$.
\end{proof}

\section{From the finite basis property to transfinite induction}\label{sect:basis-to-ti}

In this section we show that $\mathbf{PA}+\mathcal K\Sigma^-_1$ proves each instance of $\mathcal{TI}(\varepsilon_0,\Pi^-_1)$. As we will see, it follows that $\mathbf{PA}+\mathcal K\Sigma^-_1$ is a proper extension of~$\mathbf{PA}$. The result of this section is a relatively straightforward consequence of the existing literature. We provide details in order to demonstrate that the argument works out with respect to formula complexity and the role of parameters.

Let us first recall the usual notation system for ordinals below~$\varepsilon_0$. According to Cantor's normal form theorem, any ordinal~$\alpha$ can be uniquely written as
\begin{equation*}
\alpha=\omega^{\alpha_0}+\dots+\omega^{\alpha_{n-1}}\quad\text{with}\quad\alpha\succeq\alpha_0\succeq\dots\succeq\alpha_{n-1},
\end{equation*}
where $\alpha=0$ arises from~$n=0$. For $\alpha\prec\varepsilon_0=\min\{\gamma\,|\,\omega^\gamma=\gamma\}$ we have $\alpha_0\prec\alpha$. Recursively, this yields finite terms that represent all ordinals below~$\varepsilon_0$. Working in~$\mathbf{PA}$, one can develop basic ordinal arithmetic in terms of the resulting notation system (see e.\,g.~\cite{pohlers-proof-theory,sommer95}). In the following we always refer to term representations rather than actual ordinals.

In the introduction we have defined a set~$\mathcal B$ of binary trees and an embeddability relation $\leq_{\mathcal B}$. To establish a connection with the ordinals below~$\varepsilon_0$, it is convenient to have a binary normal form: If $\alpha\succ 0$ has Cantor normal form as above, we write
\begin{equation*}
\alpha=_{\operatorname{NF}}\omega^\beta+\gamma\qquad\text{for $\beta=\alpha_0$ and $\gamma=\omega^{\alpha_1}+\dots+\omega^{\alpha_{n-1}}$}.
\end{equation*}
Note that $\beta$ and $\gamma$ can be seen as proper subterms of~$\alpha$. The following construction is well-known (cf.~\cite[\S~12]{takeuti-proof-theory}).

\begin{definition}[$\mathbf{PA}$]\label{def:quasi-emb-eps-B}
We construct a function $f:\varepsilon_0\to\mathcal B$ by setting
\begin{equation*}
f(\alpha)=\begin{cases}
\circ & \text{if $\alpha=0$},\\
\circ(f(\beta),f(\gamma)) & \text{if $\alpha=_{\operatorname{NF}}\omega^\beta+\gamma$,}
\end{cases}
\end{equation*}
which amounts to a recursion over term representations of ordinals.
\end{definition}

Concerning the formalization in $\mathbf{PA}$, we note that $f$ is primitive recursive. Hence $f$ is $\mathbf{PA}$-provably total. In particular, the graph of $f$ is $\Delta^0_1$-definable in~$\mathbf{PA}$. The following folklore result shows that $f$ satisfies the definition of a quasi embedding.

\begin{lemma}[$\mathbf{PA}$]\label{lem_emb-eps-B}
For $\alpha,\beta\prec\varepsilon_0$, the inequality $f(\alpha)\leq_{\mathcal B}f(\beta)$ implies $\alpha\preceq\beta$.
\end{lemma}
\begin{proof}
Define a height function $h:\varepsilon_0\to\mathbb N$ by recursion over terms, setting
\begin{equation*}
h(\alpha)=\begin{cases}
0 & \text{if $\alpha=0$},\\
\max\{h(\gamma),h(\delta)\}+1 & \text{if $\alpha=_{\operatorname{NF}}\omega^\gamma+\delta$.}
\end{cases}
\end{equation*}
The claim from the lemma can now be verified by induction over $h(\beta)$. For $\alpha=0$ the implication holds because $\alpha\preceq\beta$ is true. In the remaining case we may write $\alpha=_{\operatorname{NF}}\omega^\gamma+\delta$. By the definition of $\leq_{\mathcal B}$, the inequality $f(\alpha)=\circ(f(\gamma),f(\delta))\leq_{\mathcal B}f(\beta)$ fails for $f(\beta)=\circ$. Hence we may also assume $\beta\succ 0$, say $\beta=_{\operatorname{NF}}\omega^{\gamma'}+\delta'$. Again by the definition of $\leq_{\mathcal B}$, the inequality
\begin{equation*}
f(\alpha)=\circ(f(\gamma),f(\delta))\leq_{\mathcal B}\circ(f(\gamma'),f(\delta'))=f(\beta)
\end{equation*}
can hold for two reasons: First assume we have $f(\gamma)\leq_{\mathcal B} f(\gamma')$ and $f(\delta)\leq_{\mathcal B} f(\delta')$. In view of $h(\gamma'),h(\delta')<h(\beta)$, the induction hypothesis yields $\gamma\preceq\gamma'$ and $\delta\preceq\delta'$. By basic ordinal arithmetic we get
\begin{equation*}
\alpha=\omega^\gamma+\delta\preceq\omega^{\gamma'}+\delta'=\beta.
\end{equation*}
Now assume $f(\alpha)\leq_{\mathcal B}f(\beta)$ holds because we have $f(\alpha)\leq_{\mathcal B}f(\gamma')$ or $f(\alpha)\leq_{\mathcal B}f(\delta')$. Inductively we get $\alpha\preceq\gamma'\preceq\omega^{\gamma'}$ or $\alpha\preceq\delta'$. Either way we have $\alpha\preceq\omega^{\gamma'}+\delta'=\beta$.
\end{proof}

In addition to the lemma itself, we will need the following standard consequence:

\begin{corollary}[$\mathbf{PA}$]
The function $f:\varepsilon_0\to\mathcal B$ is injective.
\end{corollary}
\begin{proof}
Consider $\alpha,\beta\prec\varepsilon_0$ with $f(\alpha)=f(\beta)$. A straightforward induction over~$\mathcal B$ shows that $\leq_{\mathcal B}$ is reflexive. Hence we have $f(\alpha)\leq_{\mathcal B}f(\beta)$ and $f(\beta)\leq_{\mathcal B}f(\alpha)$. By the previous lemma this implies $\alpha\preceq\beta$ and $\beta\preceq\alpha$. Since the order relation on the ordinals is antisymmetric, we obtain $\alpha=\beta$.
\end{proof}

We can now show that the finite basis property implies transfinite induction. The converse implication will be established in Section~\ref{ref:refl-to-fin-base}.

\begin{proposition}\label{prop:fin-base-to-TI}
Each instance of $\mathcal{TI}(\varepsilon_0,\Pi^-_1)$ can be proved in $\mathbf{PA}+\mathcal K\Sigma^-_1$.
\end{proposition}
\begin{proof}
Working in $\mathbf{PA}+\mathcal K\Sigma^-_1$, we establish $\mathcal{TI}(\varepsilon_0,\psi)$ for a given $\Pi^0_1$-formula $\psi$ with a single free variable. For this purpose we consider the formula
\begin{equation*}
\varphi(t):\equiv t\in\mathcal B\land\exists_{\alpha\prec\varepsilon_0}(f(\alpha)=t\land\neg\psi(\alpha)),
\end{equation*}
where $f:\varepsilon_0\to\mathcal B$ is the function from Definition~\ref{def:quasi-emb-eps-B}. Since the graph of $f$ is \mbox{$\Delta^0_1$-}definable in $\mathbf{PA}$, we see that $\varphi(t)$ is (provably equivalent to) a $\Sigma^0_1$-formula with the single free variable~$t$. Hence we may use $\mathcal K\varphi$ to get (a code for) a finite set $a\subseteq\mathcal B$ that satisfies
\begin{equation*}
\forall_{s\in a}\varphi(s)\land\forall_{t\in\mathcal B}(\varphi(t)\rightarrow\exists_{s\in a}s\leq_{\mathcal B}t).
\end{equation*}
First assume that $a$ is empty. Then $\exists_{s\in a}s\leq_{\mathcal B}t$ fails for all $t\in\mathcal B$, so that the second conjunct enforces~$\forall_{t\in\mathcal B}\neg\varphi(t)$. Given $\alpha\prec\varepsilon_0$, it is straightforward to see that $\neg\varphi(t)$ for $t:=f(\alpha)\in\mathcal B$ implies $\psi(\alpha)$. We thus have $\forall_{\alpha\prec\varepsilon_0}\psi(\alpha)$, which is the conclusion of~$\mathcal{TI}(\varepsilon_0,\psi)$. Now assume that the finite set $a\subseteq\mathcal B$ is non-empty. Due to $\forall_{s\in a}\varphi(s)$, we see that $a$ is contained in the range of~$f$. Also recall that $f$ is injective. By induction on the cardinality of~$a$, one can infer that there is an ordinal $\gamma\prec\varepsilon_0$ with
\begin{equation*}
f(\gamma)\in a\land\forall_{\delta\prec\gamma} f(\delta)\notin a.
\end{equation*}
Given an ordinal $\gamma$ with this property, we now establish
\begin{equation*}
\forall_{\beta\prec\gamma}\psi(\beta)\land\neg\psi(\gamma),
\end{equation*}
which implies that $\mathcal{TI}(\varepsilon_0,\psi)$ holds because its antecedent fails. Aiming at the first conjunct, we consider an ordinal~$\beta\prec\gamma$. If $\psi(\beta)$ was false, then $\varphi(t)$ would hold for~$t:=f(\beta)\in\mathcal B$. Since $a\subseteq\mathcal B$ witnesses the conclusion of $\mathcal K\varphi$, we would get an element $s\in a$ with $s\leq_{\mathcal B}t$. Writing $s=f(\delta)$ with $\delta\prec\varepsilon_0$, we could invoke Lemma~\ref{lem_emb-eps-B} to conclude $\delta\preceq\beta\prec\gamma$. By the above this would imply $s=f(\delta)\notin a$, which yields the desired contradiction. To establish the second conjunct we observe that $f(\gamma)\in a$ implies~$\varphi(f(\gamma))$. According to the definition of~$\varphi$, this means that there is an ordinal $\alpha\prec\varepsilon_0$ with $f(\alpha)=f(\gamma)$ and $\neg\psi(\alpha)$. Since~$f$ is injective we get $\alpha=\gamma$ and thus $\neg\psi(\gamma)$, as required.
\end{proof}

According to Gentzen's ordinal analysis~\cite{gentzen36}, the consistency of Peano arithmetic is provable in $\mathbf{PA}+\mathcal{TI}(\varepsilon_0,\Pi^-_1)$. A detailed proof of a stronger result can be found in the next section. Together with Proposition~\ref{prop:fin-base-to-TI} and G\"odel's theorem, it follows that $\mathbf{PA}+\mathcal K\Sigma^-_1$ is a proper extension of~$\mathbf{PA}$.

\section{From transfinite induction to reflection}\label{sect:ti-to-rfn}

Working over~$\mathbf{PA}$, we show that $\mathcal{TI}(\varepsilon_0,\Pi^-_1)$ implies $\operatorname{Rfn}_{\mathbf{PA}}(\Sigma^0_2)$. The converse direction will be established in Section~\ref{ref:refl-to-fin-base}. The result is rather similar to one by Kreisel and L\'evy~\cite{kreisel68}, who show that induction with parameters corresponds to uniform reflection. The author has found no reference for the parameter-free case. As we will see, the connection with reflection implies that $\mathbf{PA}+\mathcal K\Sigma^-_1$ is not contained in any consistent extension of $\mathbf{PA}$ by a computably enumerable set of $\Pi^0_2$-sentences.

As preparation, we review the ordinal analysis of Peano arithmetic and its formalization in~$\mathbf{PA}$ itself. First note that we cannot formalize the usual soundness argument by induction over formal proofs, since there is no arithmetical truth definition that would cover all relevant formulas (due to Tarski~\cite{tarski36}). Even when we restrict attention to theorems of restricted complexity, their proofs may involve detours through more complex lemmata. The method of cut elimination aims to remove such detours in order to permit a soundness argument that is based on partial truth definitions (cf.~\cite[Section~I.1(d)]{hajek91}). However, it is not immediately possible to eliminate complex lemmata from proofs in Peano arithmetic, which may use complex instances of induction in an essential way. To resolve this problem, ordinal analysis transforms the usual finite proofs into infinite proof trees: In the realm of infinite proofs, induction can be deduced from axioms of low complexity, so that cut elimination becomes possible. Soundness can then be proved by transfinite induction over the rank of infinite proof trees.

Our ordinal analysis works with proofs in a Tait-style sequence calculus. In particular, this means that all formulas are in negation normal form, and that negation is a defined operation based on Morgan's laws. Each node in a proof tree deduces a sequent, i.\,e.~a finite set $\Gamma=\{\varphi_0,\dots,\varphi_{n-1}\}$ of formulas. The latter is to be interpreted as the disjunction $\bigvee\Gamma=\varphi_0\lor\dots\lor\varphi_{n-1}$. In the context of sequents we write $\Gamma,\varphi$ for $\Gamma\cup\{\varphi\}$. Detours in proofs are implemented via the cut rule
\begin{equation*}
\begin{bprooftree}
\AxiomC{$\Gamma,\varphi$}
\AxiomC{$\Gamma,\neg\varphi$}
\RightLabel{,}
\BinaryInfC{$\Gamma$}
\end{bprooftree}
\end{equation*}
which has the following intuitive significance: In order to show $\bigvee\Gamma$, it suffices to
\begin{itemize}
\item prove a lemma~$\varphi$ (more precisely, the left premise proves $\bigvee\Gamma\lor\varphi$) and to
\item prove that $\varphi$ implies $\bigvee\Gamma$ (i.\,e.~to prove $\bigvee\Gamma\lor\neg\varphi$, as in the right premise).
\end{itemize}
The crucial feature of the infinite proof system is the $\omega$-rule
\begin{equation*}
\begin{bprooftree}
\AxiomC{$\Gamma,\varphi(0)$}
\AxiomC{$\Gamma,\varphi(1)$}
\AxiomC{$\cdots$}
\RightLabel{,}
\TrinaryInfC{$\Gamma,\forall_n\varphi(n)$}
\end{bprooftree}
\end{equation*}
which allows to infer $\forall_n\varphi(n)$ if there is a proof of $\varphi(n)$ for each numeral~$n$. Induction can be derived from the $\omega$-rule, since
\begin{equation*}
\varphi(0)\land\forall_m(\varphi(m)\rightarrow\varphi(m+1))\rightarrow\varphi(n)
\end{equation*}
has a straightforward proof for each number~$n$. It follows that any finite proof in Peano arithmetic can be translated (or ``embedded") into the infinite system.

It is not immediately clear how infinite proof trees can be formalized in Peano arithmetic. In the following we recall a very elegant approach due to Buchholz~\cite{buchholz91} (see his paper for all missing details): The idea is to work with a set~$\mathbf Z^*$ of finite terms. Each term names an infinite proof by specifying its role in the cut elimination process. Specifically, each finite proof~$d$ in Peano arithmetic gives rise to a constant symbol~$[d]\in\mathbf Z^*$, which denotes the translation of~$d$ into the infinite system. For each term $h\in\mathbf Z^*$ there is a term $Eh\in\mathbf Z^*$ that names the proof that results from~$h$ by a single application of cut elimination. The intermediate steps of cut elimination give rise to auxiliary function symbols. By primitive recursion over terms one can define an ordinal~$\mathfrak o(h)\prec\varepsilon_0$ that bounds the rank of the proof tree represented by~$h$; for example, the well-known fact that cut elimination leads to an exponential increase of the ordinal rank suggests the recursive clause $\mathfrak o(E h)=\omega^{\mathfrak o(h)}$. Also by recursion over terms, one can determine the end sequent~$\mathfrak e(h)$, the last rule~$\mathfrak r(h)$, the cut rank~$\mathfrak d(h)$, and terms $\mathfrak s(h,n)\in\mathbf Z^*$ that denote the immediate subtrees of the proof tree that is represented by~$h$. Working in~$\mathbf{PA}$ (or even in~$\mathbf{PRA}$), one can show that the term system~$\mathbf Z^*$ is ``locally correct" (see~\cite[Theorem~3.8]{buchholz91}); in particular this means that we have $\mathfrak o(\mathfrak s(h,n))\prec\mathfrak o(s)$, except when $\mathfrak r(s)$ signifies an axiom. To ensure ``global correctness", one needs transfinite induction up to~$\varepsilon_0$, which is not available in~$\mathbf{PA}$. In the sequel we abbreviate
\begin{equation*}
h\vdash^\alpha_0\Gamma\quad:\Leftrightarrow\quad h\in\mathbf Z^*\land\mathfrak o(h)=\alpha\land\mathfrak d(h)=0\land\mathfrak e(h)\subseteq\Gamma.
\end{equation*}
Intuitively, this asserts that~$h$ is a cut-free infinite proof tree with rank~$\alpha$ and end sequent~$\Gamma$ (note that $\bigvee\mathfrak e(h)$ implies $\bigvee\Gamma$). Crucially, the relation $h\vdash^\alpha_0\Gamma$ is primitive recursive and hence $\Delta^0_1$-definable in~$\mathbf{PA}$. This implies that
\begin{equation*}
\mathbf Z^*\vdash^\alpha_0\Gamma\quad:\Leftrightarrow\quad\exists_{h\in\mathbf Z^*}h\vdash^\alpha_0\Gamma
\end{equation*}
is a $\Sigma^0_1$-formula with parameters~$\alpha$ and $\Gamma$. We can now show the promised result:

\begin{proposition}\label{prop:TI-implies-Rfn}
Each instance of $\operatorname{Rfn}_{\mathbf{PA}}(\Sigma^0_2)$ can be proved in $\mathbf{PA}+\mathcal{TI}(\varepsilon_0,\Pi^-_1)$.
\end{proposition}
\begin{proof}
Consider a closed~$\Sigma^0_2$-formula~$\varphi$. Working in $\mathbf{PA}+\mathcal{TI}(\varepsilon_0,\Pi^-_1)$, we assume that we have $\operatorname{Pr}_{\mathbf{PA}}(\varphi)$. In order to establish $\operatorname{Rfn}_{\mathbf{PA}}(\varphi)$, we need to derive~$\varphi$. We use Buchholz' formalization of ordinal analysis, as discussed above. By embedding and cut elimination (cf.~\cite[Definitions~3.4 and~3.7]{buchholz91}), the assumption $\operatorname{Pr}_{\mathbf{PA}}(\varphi)$ implies
\begin{equation*}
\exists_{\alpha\prec\varepsilon_0}\mathbf Z^*\vdash^\alpha_0\{\varphi\}.
\end{equation*}
Write $\Gamma\subseteq\{\varphi\}\cup\Pi^-_1$ to express that $\Gamma$ is a sequent that consists of $\Pi^0_1$-sentences and (possibly) the formula~$\varphi$. The statement that $\Gamma$ contains a true $\Pi^0_1$-sentence can be expressed by a $\Pi^0_1$-formula $\operatorname{Tr}_{\Pi^-_1}(\Gamma)$ (cf.~\cite[Theorem~I.1.75]{hajek91}). Aiming at a contradiction, we assume that $\varphi$ is false. Under this assumption we will derive
\begin{equation*}
\forall_{\alpha\prec\varepsilon_0}\forall_{\Gamma}(\Gamma\subseteq\{\varphi\}\cup\Pi^-_1\land\mathbf Z^*\vdash^\alpha_0\Gamma\rightarrow\operatorname{Tr}_{\Pi^-_1}(\Gamma)),
\end{equation*}
arguing by transfinite induction on~$\alpha\prec\varepsilon_0$. Note that the sentence~$\varphi$ is represented by a fixed numeral. Hence~$\alpha$ is the only free variable of the induction formula, and the induction is covered by the scheme $\mathcal{TI}(\varepsilon_0,\Pi^-_1)$. Once the induction is carried out, it is starightforward to derive the desired contradiction: By the above we have $\mathbf Z^*\vdash^\alpha_0\{\varphi\}$ for some $\alpha\prec\varepsilon_0$. However, we cannot have $\operatorname{Tr}_{\Pi^-_1}(\{\varphi\})$, since $\varphi$ was assumed to be false (note that this covers both $\varphi\in\Pi^0_1\subseteq\Sigma^0_2$ and $\varphi\in\Sigma^0_2\backslash\Pi^0_1$). It remains to carry out the induction. In the step we consider a sequent $\Gamma\subseteq\{\varphi\}\cup\Pi^-_1$ and assume $h\vdash^\alpha_0\Gamma$ for some $h\in\mathbf Z^*$. We distinguish cases according to the last rule~$\mathfrak r(h)$. Note that this cannot be a cut, since $h\vdash^\alpha_0\Gamma$ entails $\mathfrak d(h)=0$. If $\mathfrak r(h)$ is an axiom, then $\mathfrak e(h)\subseteq\Gamma$ contains a true literal (cf.~\cite[Definition~2.2]{buchholz91}). To complete the proof, we consider the introduction of a quantifier; the introduction of a propositional connective is similar and simpler. First assume that $h$ ends with an $\omega$-rule, which introduces a formula $\forall_n\theta(n)\in\Gamma$. Due to $\Gamma\subseteq\{\varphi\}\cup\Pi^-_1$ we see that $\forall_n\theta(n)$ must be a $\Pi^0_1$-sentence. Local correctness (see~\cite[Theorem~3.8]{buchholz91}) yields
\begin{equation*}
\mathbf Z^*\vdash^{\mathfrak o(\mathfrak s(h,n))}_0\Gamma,\theta(n)\quad\text{with}\quad\mathfrak o(\mathfrak s(h,n))\prec \mathfrak o(h)=\alpha
\end{equation*}
for all $n\in\mathbb N$. The induction hypothesis implies that each sequent $\Gamma,\theta(n)$ contains a true $\Pi^0_1$-sentence. Hence we get such a sentence in $\Gamma$, or all instances $\theta(n)$ are true. In the latter case, it follows that $\Gamma$ contains the true $\Pi^0_1$-sentence $\forall_n\theta(n)$. Finally, assume that $\mathfrak r(h)$ introduces an existential formula~$\exists_n\psi(n)$. In view of $\Gamma\subseteq\{\varphi\}\cup\Pi^-_1$ we must have $\exists_n\psi(n)\equiv\varphi$ (note that~\cite{buchholz91} does not work with bounded quantifiers but treats primitive recursive relations as atomic). By local correctness there is some existential witness~$k\in\mathbb N$ such that we have
\begin{equation*}
\mathbf Z^*\vdash^{\mathfrak o(\mathfrak s(h,0))}_0\Gamma,\psi(k)\quad\text{with}\quad\mathfrak o(\mathfrak s(h,0))\prec \mathfrak o(h)=\alpha.
\end{equation*}
The induction hypothesis yields a true $\Pi^0_1$-sentence in $\Gamma,\psi(k)$. To establish $\operatorname{Tr}_{\Pi^-_1}(\Gamma)$ it suffices to show that $\psi(k)$ cannot be true: if it was, then $\varphi\equiv\exists_n\psi(n)$ would be true as well, which contradicts our assumption.
\end{proof}

The following proof is similar to one by Kreisel and L\'evy~\cite[\S~8]{kreisel68} (see~\cite[Lemma~2]{beklemishev97-local} for an argument that takes the formula complexity into account).

\begin{corollary}\label{cor:Sigma2-essential}
There is no computably enumerable set $\Psi$ of $\Pi^0_2$-sentences such that $\mathbf{PA}+\Psi$ is consistent and contains $\mathbf{PA}+\mathcal K\Sigma^-_1$. In particular, the latter is a proper extension of $\mathbf{PA}$.
\end{corollary}
\begin{proof}
Consider a computably enumerable set $\Psi$ of $\Pi^0_2$-sentences such that $\mathbf{PA}+\Psi$ proves each instance of~$\mathcal K\Sigma^-_1$. We need to show that $\mathbf{PA}+\Psi$ is inconsistent. According to~\cite[Theorem~4]{lindstroem-interpretability84}, there is a single $\Pi^0_2$-sentence~$\psi$ such that $\mathbf{PA}+\psi$ is a $\Sigma^0_2$-conservative extension of $\mathbf{PA}+\Psi$. In view of conservativity, it suffices to show that $\mathbf{PA}+\psi$ is inconsistent. By Propositions~\ref{prop:fin-base-to-TI} and~\ref{prop:TI-implies-Rfn} we have
\begin{equation*}
\mathbf{PA}+\operatorname{Rfn}_{\mathbf{PA}}(\Sigma^0_2)\,\subseteq\,\mathbf{PA}+\mathcal{TI}(\varepsilon_0,\Pi^-_1)\,\subseteq\,\mathbf{PA}+\mathcal K\Sigma^-_1\,\subseteq\,\mathbf{PA}+\Psi\,\subseteq\,\mathbf{PA}+\psi.
\end{equation*}
Hence we can invoke local $\Sigma^0_2$-reflection to get
\begin{equation*}
\mathbf{PA}+\psi\vdash\operatorname{Pr}_{\mathbf{PA}}(\neg\psi)\to\neg\psi.
\end{equation*}
The contrapositive yields $\mathbf{PA}+\psi\vdash\neg\operatorname{Pr}_{\mathbf{PA}}(\neg\psi)$. This means that $\mathbf{PA}+\psi$ proves its own consistency, so that it is inconsistent by G\"odel's theorem.
\end{proof}

Since any true $\Sigma^0_2$-sentence follows from a true $\Pi^0_1$-sentence, there is a set $\Xi$ of $\Pi^0_1$-sentences such that $\mathbf{PA}+\Xi$ is consistent and contains~$\mathbf{PA}+\mathcal K\Sigma^-_1$. The corollary tells us that $\Xi$ cannot be computably enumerable.

\section{A primitive recursive reification}\label{sect:reification}

In the rest of this paper we complete the proof that $\mathcal K\Sigma^-_1$, $\mathcal{TI}(\varepsilon_0,\Pi^-_1)$ and $\operatorname{Rfn}_{\mathbf{PA}}(\Sigma^0_2)$ are equivalent over~$\mathbf{PA}$. The present section is concerned with a technical result that will be crucial for this purpose.

Write $\operatorname{Bad}(\mathcal B)$ for the set of non-empty finite bad sequences in~$\mathcal B$. We want to construct a primitive recursive function $r:\operatorname{Bad}(\mathcal B)\rightarrow\varepsilon_0$ such that we have
\begin{equation*}
r(\langle t_0,\dots,t_n,t_{n+1}\rangle)\prec r(\langle t_0,\dots,t_n\rangle)
\end{equation*}
whenever $\langle t_0,\dots,t_{n+1}\rangle$ is an element of $\operatorname{Bad}(\mathcal B)$, provably in~$\mathbf{PA}$. Such a function is called a reification. It ensures that $\mathcal B$ is a well partial order with maximal order type at most (and in fact equal to)~$\varepsilon_0$.

As mentioned in the introduction, the result that $\mathcal B$ has maximal order type~$\varepsilon_0$ is due to de Jongh and Schmidt. Experience shows that maximal order types can be witnessed by effective reifications. For the case of finite (and in particular~binary) trees this has been established by M.~Rathjen and A.~Weiermann~\cite[Section~2]{rathjen-weiermann-kruskal}. Unfortunately, we cannot simply cite their result: In~\cite{rathjen-weiermann-kruskal} it is shown that $\mathbf{ACA_0}$ proves the existence of a reification; however, it is not entirely trivial to see that the constructed reification is (primitive) recursive. In the rest of this section we verify this fact in detail. Some readers may prefer to skip this verification and to continue with the applications in the next section. We point out that the following presentation is influenced by the more general construction in~\cite{hasegawa94}.

The reification of $\mathcal B$ will depend on reifications of various other orders. In the context of first order arithmetic it helps to think of these orders as types, which are represented by finite expressions.

\begin{definition}[$\mathbf{PA}$]\label{def:types}
The following recursive clauses generate a collection of types and a subcollection of indecomposable types:
\begin{enumerate}[label=(\roman*)]
\item The symbols $\mathfrak B$ and $\mathfrak E$ are indecomposable types.
\item If $A,B$ are types, then~$A+B$ is a type.
\item If $A,B$ are indecomposable types, then $A\times B$ is an indecomposable type.
\item If $A$ is any type, then $A^*$ is an indecomposable type.
\end{enumerate}
\end{definition}

Note that it is not allowed to form types such as $(A+B)\times C$, since $A+B$ is not indecomposable. This will become important in the proof of Proposition~\ref{prop:subalgs-smaller}. The elements of our orders are represented by terms of the corresponding types. To obtain primitive recursive constructions, it is crucial to work with terms of all types simultaneously. For example, it is neither possible nor necessary to construct all terms of type~$A$ before one constructs a term of type~$A^*$. We do not specify terms of type $\mathfrak E$, because the latter is supposed to represent the empty order.

\begin{definition}[$\mathbf{PA}$]
The following recursive clauses generate a collection of terms. We simultaneously specify the types of these terms:
\begin{enumerate}[label=(\roman*)]
\item Each binary tree $t\in\mathcal B$ is a term of type~$\mathfrak B$.
\item If $a$ is a term of type $A$ and $B$ is a type, then $\iota_0^Ba$ is a term of type $A+B$. If $b$ is a term of type $B$ and $A$ is a type, then $\iota_1^Ab$ is a term of type $A+B$.
\item If $a$ and $b$ are terms of types~$A$ and $B$, then $\langle a,b\rangle$ is a term of type $A\times B$.
\item If $a_0,\dots,a_{n-1}$ have type $A$, then $\langle a_0,\dots,a_{n-1}\rangle_A$ is a term of type~$A^*$.
\end{enumerate}
Note that~(iii) does only apply when $A$ and $B$ are indecomposable.
\end{definition}

One readily constructs a G\"odel numbering $\#$ with the monotonicity properties
\begin{gather*}
\#s,\#t<\#{\circ(s,t)}\text{ for $s,t\in\mathfrak B$},\qquad\#a<\#\iota_0^Ba,\qquad\#b<\#\iota_1^Ab,\\
\#a,\#b<\#\langle a,b\rangle,\qquad\#a_0,\#\langle a_1,\dots,a_n\rangle_A<\#\langle a_0,\dots,a_n\rangle_A.
\end{gather*}
We will use this G\"odel numbering to construct primitive recursive functions by course-of-values recursion. Binary functions can be constructed with the help of the Cantor pairing function, which is monotone in both components. For example, the following definition decides $a\leq_A a'$ by recursion over the code of $\langle\#a,\#a'\rangle$.

\begin{definition}[$\mathbf{PA}$]
The relation $a\leq_A a'$ between terms $a$ and $a'$ of the same type~$A$ is generated by the following recursive clauses (i.\,e.~it is the smallest relation that satisfies them):
\begin{enumerate}[label=(\roman*)]
\item If $s\leq_{\mathcal B} t$, then $s\leq_{\mathfrak B} t$.
\item If $a\leq_A a'$, then $\iota_0^Ba\leq_{A+B}\iota_0^Ba'$. If $b\leq_B b'$, then $\iota_1^Ab\leq_{A+B}\iota_1^Ab'$.
\item If $a\leq_A a'$ and $b\leq_B b'$, then $\langle a,b\rangle\leq_{A\times B}\langle a',b'\rangle$.
\item If there is a strictly increasing $f:\{0,\dots,m-1\}\rightarrow\{0,\dots,n-1\}$ such that $a_i\leq_A a'_{f(i)}$ holds for all $i<m$, then $\langle a_0,\dots,a_{m-1}\rangle_A\leq_{A^*}\langle a'_0,\dots,a'_{n-1}\rangle_A$.
\end{enumerate}
\end{definition}

Let us record the expected property:

\begin{lemma}[$\mathbf{PA}$]
Each relation $\leq_A$ is a partial order on the terms of type~$A$.
\end{lemma}
\begin{proof}
First check $a\leq_A a$ by induction over~$\#a$, simultaneously for all types~$A$. Then use induction over $\#a+\#a'$ to verify that $a\leq_A a'$ and $a'\leq_A a$ imply $a=a'$. Finally, show $a\leq_A a'\,\&\, a'\leq_A a''\Rightarrow a\leq_Aa''$ by induction over $\#a+\#a'+\#a''$.
\end{proof}

From now on we write $a\in A$ to express that $a$ is a term of type~$A$. Despite this notation, one should keep in mind that $A$ is a finite expression rather than an infinite set. The following provides a substitute for the ``missing" types~$A\times B$.

\begin{definition}[$\mathbf{PA}$]
For arbitrary types~$A$ and $B$ we recursively define a type $A\otimes B$ and terms $[a,b]\in A\otimes B$ for all $a\in A$ and $b\in B$: First put
\begin{equation*}
A\otimes B=A\times B\quad\text{and}\quad[a,b]=\langle a,b\rangle\quad\text{when $A,B$ are indecomposable}.
\end{equation*}
Now consider $A=C+D$ and an arbitrary~$B$. To save parentheses, we assume that $\otimes$ binds stronger than~$+$. We then define
\begin{equation*}
(C+D)\otimes B=C\otimes B+D\otimes B\quad\text{and}\quad[\iota_0^Dc,b]=\iota_0^{D\otimes B}[c,b],\,[\iota_1^Cd,b]=\iota_1^{C\otimes B}[d,b].
\end{equation*}
For indecomposable~$A$ and $B=C+D$ we set
\begin{equation*}
A\otimes(C+D)=A\otimes C+A\otimes D\quad\text{and}\quad[a,\iota_0^Dc]=\iota_0^{A\otimes D}[a,c],\,[a,\iota_1^Cd]=\iota_1^{A\otimes C}[a,d].
\end{equation*}
\end{definition}

The following is readily checked by induction on $\#a+\#a'+\#b+\#b'$.

\begin{lemma}[$\mathbf{PA}$]
We have
\begin{equation*}
[a,b]\leq_{A\otimes B}[a',b']\qquad\Leftrightarrow\qquad a\leq_A a'\text{ and }b\leq_B b'
\end{equation*}
for arbitrary terms $a,a'\in A$ and $b,b'\in B$.
\end{lemma}

For $a\in A$ we will abbreviate
\begin{equation*}
a'\in A_a\quad :\Leftrightarrow\quad a'\in A\text{ and }a\not\leq_A a'.
\end{equation*}
The sets $A_a$ are important for the analysis of maximal order types, because they contain all elements that can follow $a$ in a bad sequence. In our setting it will be important to have a quasi embedding of $A_a$ into a suitable type $A(a)$. To save parentheses we agree on $A\otimes B\otimes C=(A\otimes B)\otimes C$ and $[a,b,c]=[[a,b],c]$. The following construction is similar to the one in~\cite[Definition~5.3 and Example~5.4]{hasegawa94}.

\begin{definition}[$\mathbf{PA}$]\label{def:type-A(a)}
By recursion over $\#a$ we define a type $A(a)$ for each $a\in A$:
\begin{enumerate}[label=(\roman*)]
\item We have $\mathfrak B(\circ)=\mathfrak E$ and $\mathfrak B(\circ(s,t))=(\mathfrak B(s)+\mathfrak B(t))^*$.
\item We have $(A+B)(\iota_0^Ba)=A(a)+B$ and $(A+B)(\iota_1^Ab)=A+B(b)$.
\item We have $(A\times B)(\langle a,b\rangle)=A(a)\otimes B+A\otimes B(b)$.
\item We have $A^*(\langle\rangle_A)=\mathfrak E$ and
\begin{equation*}
A^*(\langle a_0,\dots,a_n\rangle_A)=A(a_0)^*+A(a_0)^*\otimes A\otimes A^*(\langle a_1,\dots,a_n\rangle_A).
\end{equation*}
\end{enumerate}
\end{definition}

As promised, we get the following quasi embeddings:

\begin{proposition}\label{prop:A_a-into-A(a)}
There is a primitive recursive function $e$ such that $\mathbf{PA}$ proves the following: For any type $A$ and terms $a\in A, b\in A_a$ we have $e_A(a,b)=e(a,b)\in A(a)$ (note that $A$ can be inferred from $a$). Furthermore we have
\begin{equation*}
e_A(a,b)\leq_{A(a)}e_A(a,b')\quad\Rightarrow\quad b\leq_A b'
\end{equation*}
for any terms $b,b'\in A_a$.
\end{proposition}
\begin{proof}
The value $e_A(a,b)$ is defined by recursion over the code of the pair~$\langle\#a,\#b\rangle$, simultaneously for all types~$A$. Once the construction of $e$ is complete, the second part of the proposition can be verified by induction on $\#a+\#b+\#b'$. In the following we distinguish cases according to the form of~$a$.

First consider $a=\circ\in\mathfrak B=A$. Since $\circ\leq_{\mathfrak B}t$ is true for any $t\in\mathfrak B$, the set $A_a$ is empty and there are no values to define. Now assume $a=\circ(s_0,s_1)\in\mathfrak B=A$. For the term $b=\circ\in\mathfrak B$ we put
\begin{equation*}
e_{\mathfrak B}(\circ(s_0,s_1),\circ)=\langle\rangle_{\mathfrak B(s_0)+\mathfrak B(s_1)}\in(\mathfrak B(s_0)+\mathfrak B(s_1))^*=\mathfrak B(\circ(s_0,s_1)).
\end{equation*}
Now assume that we have $b=\circ(t_0,t_1)\in\mathfrak B$. The condition $b\in A_a$ amounts to $\circ(s_0,s_1)\not\leq_{\mathfrak B}\circ(t_0,t_1)$, which yields $s_0\not\leq_{\mathfrak B}t_0$ or $s_1\not\leq_{\mathfrak B}t_1$. Let us assume that we have $s_0\not\leq_{\mathfrak B}t_0$, which amounts to $t_0\in\mathfrak B_{s_0}$. We may then refer to the recursively defined value
\begin{equation*}
e_{\mathfrak B}(s_0,t_0)\in\mathfrak B(s_0).
\end{equation*}
More formally, the recursive definition of $e_A(a,b)$ and the inductive verification of $e_A(a,b)\in A(a)$ should be separated. In order to do so, we can agree on a default value for the hypothetical case that the decidable property $e_{\mathfrak B}(s_0,t_0)\in\mathfrak B(s_0)$ fails; the induction shows that the default value is never called. By $\circ(s_0,s_1)\not\leq_{\mathfrak B}\circ(t_0,t_1)$ we also have $\circ(s_0,s_1)\not\leq_{\mathfrak B}t_1$, which amounts to $t_1\in\mathfrak B_{\circ(s_0,s_1)}$ and provides
\begin{equation*}
e_{\mathfrak B}(\circ(s_0,s_1),t_1)\in\mathfrak B(\circ(s_0,s_1))=(\mathfrak B(s_0)+\mathfrak B(s_1))^*.
\end{equation*}
Let us agree to write $c_0\star\langle c_1,\dots,c_n\rangle_C:=\langle c_0,c_1,\dots,c_n\rangle_C\in C^*$ for terms~$c_0,\dots,c_n$ of a type~$C$. We can now state our recursive clause as
\begin{equation*}
e_{\mathfrak B}(\circ(s_0,s_1),\circ(t_0,t_1))=\begin{cases}
\iota_0^{\mathfrak B(s_1)}e_{\mathfrak B}(s_0,t_0)\star e_{\mathfrak B}(\circ(s_0,s_1),t_1) & \text{if $s_0\not\leq_{\mathfrak B}t_0$},\\[1ex]
\iota_1^{\mathfrak B(s_0)}e_{\mathfrak B}(s_1,t_1)\star e_{\mathfrak B}(\circ(s_0,s_1),t_0) & \text{otherwise}.
\end{cases}
\end{equation*}
To explain the second case we recall that $s_1\not\leq_{\mathfrak B}t_1$ must hold if $s_0\not\leq_{\mathfrak B}t_0$ fails.

Before we state the other recursive clauses, let us verify that the second part of the proposition holds for $A=\mathfrak B$. As above we write $a=\circ(s_0,s_1)$. In the case of the term $b'=\circ$ we observe
\begin{equation*}
e_{\mathfrak B}(a,b)\leq_{\mathfrak B(s)}e_{\mathfrak B}(a,b)=\langle\rangle_{\mathfrak B(s_0)+\mathfrak B(s_1)}\quad\Rightarrow\quad e_{\mathfrak B}(a,b)=\langle\rangle_{\mathfrak B(s_0)+\mathfrak B(s_1)}.
\end{equation*}
The consequent of this implication can only hold for $b=\circ$. In this case $b\leq_{\mathfrak B}b'$ is satisfied for any $b'\in\mathfrak B$. Hence it remains to consider terms of the form $b=\circ(t_0,t_1)$ and $b'=\circ(t_0',t_1')$. In general we have
\begin{equation*}
c\star\sigma\leq_{C^*}c'\star\sigma'\quad\Leftrightarrow\quad c\star\sigma\leq_{C^*}\sigma'\text{ or }(c\leq_C c'\text{ and }\sigma\leq_{C^*}\sigma').
\end{equation*}
First assume that $e_{\mathfrak B}(s,b)\leq_{\mathfrak B(s)}e_{\mathfrak B}(s,b')$ holds because of $e_{\mathfrak B}(s,b)\leq_{\mathfrak B(s)}e_{\mathfrak B}(s,t_i')$. Then the induction hypothesis yields $b\leq_{\mathfrak B}t_i'$, which implies $b\leq_{\mathfrak B}\circ(t_0',t_1')=b'$. Now assume we have $e_{\mathfrak B}(s,b)\leq_{\mathfrak B(s)}e_{\mathfrak B}(s,b')$ because there are $i,j\in\{0,1\}$ with
\begin{align*}
\iota_i^{\mathfrak B(s_{1-i})}e_{\mathfrak B}(s_i,t_i)&\leq_{\mathfrak B(s_0)+\mathfrak B(s_1)}\iota_j^{\mathfrak B(s_{1-j})}e_{\mathfrak B}(s_j,t_j'),\\[1ex]
e_{\mathfrak B}(s,t_{1-i})&\leq_{\mathfrak B(s)}e_{\mathfrak B}(s,t_{1-j}').
\end{align*}
The first inequality can only hold for $i=j$. It yields $e_{\mathfrak B}(s_i,t_i)\leq_{\mathfrak B(s_i)}e_{\mathfrak B}(s_i,t_i')$, which implies $t_i\leq_{\mathfrak B}t_i'$ by induction hypothesis. From the second inequality we can infer $t_{1-i}\leq_{\mathfrak B}t_{1-i}'$. Together we get $b=\circ(t_0,t_1)\leq_{\mathfrak B}\circ(t_0',t_1')=b'$, as desired.

Sum and product types are considerably easier to handle. We only state the recursive clauses and leave all verifications to the reader:
\begin{gather*}
\begin{aligned}
e_{A+B}(\iota_0^Ba,\iota_0^Ba')&=\iota_0^B e_A(a,a'),\quad &  e_{A+B}(\iota_0^Ba,\iota_1^Ab')&=\iota_1^{A(a)}b',\\[1ex]
e_{A+B}(\iota_1^Ab,\iota_0^Ba')&=\iota_0^{B(b)}a', & e_{A+B}(\iota_1^Ab,\iota_1^Ab')&=\iota_1^A e_B(b,b').
\end{aligned}\\[1ex]
e_{A\times B}(\langle a,b\rangle,\langle a',b'\rangle)=\begin{cases}
\iota_0^{A\otimes B(b)}[e_A(a,a'),b'] & \text{if $a\not\leq_A a'$},\\[1ex]
\iota_1^{A(a)\otimes B}[a',e_B(b,b')] & \text{otherwise}.
\end{cases}
\end{gather*}

Finally, we consider the case of a type~$A^*$. For $a=\langle\rangle_A\in A^*$ it suffices to observe that $(A^*)_a$ is empty, since $\langle\rangle_A\leq_{A^*}\tau$ holds for any $\tau\in A^*$. Now consider a term of the form $a=a_0\star\sigma\in A^*$. We write $b=\langle b_0,\dots,b_{n-1}\rangle_A\in(A^*)_a$ and distinguish two cases. If we have $a_0\not\leq_A b_i$ for all $i<n$, then we set
\begin{equation*}
e_{A^*}(a,b)=\iota_0^{A(a_0)^*\otimes A\otimes A^*(\sigma)}\langle e_A(a_0,b_0),\dots,e_A(a_0,b_{n-1})\rangle_{A(a_0)}.
\end{equation*}
Note that this is an element of $A(a_0)^*+A(a_0)^*\otimes A\otimes A^*(\sigma)=A^*(a)$, as required. Otherwise we fix the smallest number $i<n$ with $a_0\leq_A b_i$. In view of $b\in(A^*)_a$ we must have $\sigma\not\leq_{A^*}\langle b_{i+1},\dots,b_{n-1}\rangle_A$. We can thus define $e_{A^*}(a,b)$  as
\begin{equation*}
\iota_1^{A(a_0)^*}[\langle e_A(a_0,b_0),\dots,e_A(a_0,b_{i-1})\rangle_{A(a_0)},b_i,e_{A^*}(\sigma,\langle b_{i+1},\dots,b_{n-1}\rangle_A)].
\end{equation*}
Using the induction hypothesis, one readily checks that $e_{A^*}(a,b)\leq_{A^*(a)}e_{A^*}(a,b')$ implies $b\leq_{A^*}b'$.
\end{proof}

Our next aim is to iterate the previous construction along bad sequences. Given a type~$A$, we write $\sigma\in\operatorname{Bad}^+(A)$ to express that $\sigma$ is a finite bad sequence in~$A$. This means that we have $\sigma=\langle a_0,\dots,a_{n-1}\rangle$ for terms $a_0,\dots,a_{n-1}\in A$ that satisfy $a_i\not\leq_A a_j$ for all $i<j<n$. If we have $\sigma\in\operatorname{Bad}^+(A)$ and $\sigma$ is different from the empty sequence~$\langle\rangle$, then we write $\sigma\in\operatorname{Bad}(A)$. For $\sigma=\langle a_0,\dots,a_{n-1}\rangle\in\operatorname{Bad}^+(A)$ we abbreviate $\sigma^\frown a=\langle a_0,\dots,a_{n-1},a\rangle$ and put
\begin{equation*}
a\in A_\sigma\qquad:\Leftrightarrow\qquad a\in A\text{ and }\sigma^\frown a\in\operatorname{Bad}(A).
\end{equation*}
The expressions $A(a)$ and $e_A(a,b)$ have only been explained for $a\in A$ and $b\in A_a$. We will see that the following definition does conform with these restrictions. In order to state the definition it is, nevertheless, helpful to realize that the primitive recursive functions $(A,a)\mapsto A(a)$ and $(A,a,b)\mapsto e_A(a,b)$ can be extended to arbitrary arguments.

\begin{definition}[$\mathbf{PA}$]\label{def:bad-seq-types}
Consider a type~$A$. For a sequence $\sigma\in\operatorname{Bad}^+(A)$ and a term $b\in A_\sigma$ we define $A[\sigma]$ and $\hat e_A(\sigma,b)$ by the recursive clauses
\begin{align*}
A[\langle\rangle]&=A, & A[\sigma^\frown a]&=A[\sigma](\hat e_A(\sigma,a)),\\
\hat e_A(\langle\rangle,b)&=b, & \hat e_A(\sigma^\frown a,b)&=e_{A[\sigma]}(\hat e_A(\sigma,a),\hat e_A(\sigma,b)).
\end{align*}
\end{definition}

In order to justify the recursion in detail, we consider $\sigma=\langle a_0,\dots,a_{n-1}\rangle$ and write $\sigma\!\restriction\!i=\langle a_0,\dots,a_{i-1}\rangle$. Then $A[\sigma\!\restriction\!i]$ and the values $\hat e_A(\sigma\!\restriction\!i,a_j)$ for $i\leq j<n$ are constructed simultaneously by recursion on~$i<n$. For $\sigma':=\sigma^\frown a_n$ with $a_n:=b$ this also explains the value $\hat e_A(\sigma,b)=\hat e_A(\sigma'\!\restriction\!n,a_n)$.

\begin{corollary}[$\mathbf{PA}$]\label{cor:bad-seq-types}
If $\sigma$ is a finite bad sequence in the type $A$, then $A[\sigma]$ is a type. For any $b\in A_\sigma$ the value $\hat e_A(\sigma,b)$ is a term of this type. Furthermore we have
\begin{equation*}
\hat e_A(\sigma,b)\leq_{A[\sigma]}\hat e_A(\sigma,b')\quad\Rightarrow\quad b\leq_A b'
\end{equation*}
for any terms $b,b'\in A_\sigma$. 
\end{corollary}
\begin{proof}
We use induction on~$\sigma$ to verify all claims simultaneously. The case of $\sigma=\langle\rangle$ is immediate. Now assume that we have $\sigma={\sigma_0}^\frown a$. The induction hypothesis tells us that $\hat e_A(\sigma_0,a)$ is a term of type~$A[\sigma_0]$. In view of Definition~\ref{def:type-A(a)} it follows that~$A[\sigma]=A[\sigma_0](\hat e_A(\sigma_0,a))$ is a type. For $b\in A_\sigma$ we have $a\not\leq_A b$, so that the induction hypothesis yields $\hat e_A(\sigma_0,a)\not\leq_{A[\sigma_0]}\hat e_A(\sigma_0,b)$. By Proposition~\ref{prop:A_a-into-A(a)} we get
\begin{equation*}
\hat e_A(\sigma,b)=e_{A[\sigma_0]}(\hat e_A(\sigma_0,a),\hat e_A(\sigma_0,b))\in A[\sigma_0](\hat e_A(\sigma_0,a))=A[\sigma].
\end{equation*}
From $\hat e_A(\sigma,b)\leq_{A[\sigma]}\hat e_A(\sigma,b')$ we can infer $\hat e_A(\sigma_0,b)\leq_{A[\sigma_0]}\hat e_A(\sigma_0,b')$, also by Proposition~\ref{prop:A_a-into-A(a)}. Then $b\leq_A b'$ follows by induction hypothesis.
\end{proof}

In order to obtain a reification, it remains to assign a suitable ordinal to each type. Let us write $\alpha\oplus\beta$ and $\alpha\otimes\beta$ for the natural (``Hessenberg") sum and product of ordinals $\alpha,\beta\prec\varepsilon_0$ (see e.\,g.~\cite[\S~4]{simpson_hilbert-basis}). In contrast to the usual operations of ordinal arithmetic, the natural variants are commutative and strictly increasing in both arguments. Ordinals of the form $\omega^\gamma$ are additively indecomposable, in the sense that $\alpha,\beta\prec\omega^\gamma$ implies $\alpha\oplus\beta\prec\omega^{\gamma}$; conversely, any additively indecomposable ordinal $\delta\neq 0$ has the form $\delta=\omega^\gamma$. For $\alpha,\beta\prec\omega_2^\gamma:=\omega^{(\omega^\gamma)}$ we have $\alpha\otimes\beta\prec\omega_2^\gamma$.

\begin{definition}[$\mathbf{PA}$]
Let us say that a type is low if it does not involve the constant symbol~$\mathfrak B$. We recursively assign an ordinal $o(A)$ to each low type~$A$:
\begin{align*}
o(\mathfrak E)&=0, & o(A+B)&=o(A)\oplus o(B),\\
o(A\times B)&=o(A)\otimes o(B), & o(A^*)&=\omega_2^{o(A)}.
\end{align*}
\end{definition}

The following is crucial for the construction of a reification.

\begin{proposition}[$\mathbf{PA}$]\label{prop:subalgs-smaller}
If $A$ is a low type and $a\in A$ is a term, then $A(a)$ is a low type and we have $o(A(a))\prec o(A)$.
\end{proposition}
\begin{proof}
As preparation we note that $A\otimes B$ is low when the same holds for $A$ and~$B$. A straightforward induction shows $o(A\otimes B)=o(A)\otimes o(B)$; for example, the distributivity property from~\cite[Lemma~4.5(8)]{simpson_hilbert-basis} accounts for the inductive verification
\begin{multline*}
o((C+D)\otimes B)=o(C\otimes B+D\otimes B)=o(C\otimes B)\oplus o(D\otimes B)=\\
=(o(C)\otimes o(B))\oplus(o(D)\otimes o(B))=(o(C)\oplus o(D))\otimes o(B)=o(C+D)\otimes o(B).
\end{multline*}
By induction on~$A$ one can show that~$o(A)$ is additively indecomposable when~$A$ is an indecomposable type. The most interesting step concerns a type~$A=B\times C$, where $B$ and $C$ are indecomposable according to Definition~\ref{def:types}. Inductively we may write $o(B)=\omega^\beta$ and $o(C)=\omega^\gamma$ (unless we have $o(A)=0$). Then
\begin{equation*}
o(B\times C)=o(B)\otimes o(C)=\omega^\beta\otimes\omega^\gamma=\omega^{\beta\oplus\gamma}
\end{equation*}
is an additively indecomposable ordinal as well. The claim of the proposition can now be verified by induction over~$\#a$, for all types~$A$ simultaneously. First consider the case of a term $\iota_0^Ba\in A+B$. The induction hypothesis tells us that $A(a)$ is low with $o(A(a))\prec o(A)$. Hence $(A+B)(\iota_0^Ba)=A(a)+B$ is low and we have
\begin{equation*}
o((A+B)(\iota_0^B))=o(A(a)+B)=o(A(a))\oplus o(B)\prec o(A)\oplus o(B)=o(A+B).
\end{equation*}
The case of $\iota_1^Ab\in A+B$ is analogous. Now consider a term $\langle a,b\rangle\in A\times B$. In view of the above, the induction hypothesis implies that $A(a)\otimes B$ is low with ordinal
\begin{equation*}
o(A(a)\otimes B)=o(A(a))\otimes o(B)\prec o(A)\otimes o(B)=o(A\times B).
\end{equation*}
In the same way we get $o(A\otimes B(b))\prec o(A\times B)$. In view of Definition~\ref{def:types}, a type of the form $A\times B$ is always indecomposable. By the above this entails that $o(A\times B)$ is an additively indecomposable ordinal. Hence we obtain
\begin{equation*}
o((A\times B)(\langle a,b\rangle))=o(A(a)\otimes B+A\otimes B(b))=o(A(a)\otimes B)\oplus o(A\otimes B(b))\prec o(A\times B).
\end{equation*}
Finally, we consider the case of a type~$A^*$. Concerning the term $\langle\rangle_A\in A^*$, we note
\begin{equation*}
o(A^*(\langle\rangle_A))=o(\mathfrak E)=0\prec\omega_2^{o(A)}=o(A^*).
\end{equation*}
Now consider a term $a\star\sigma\in A^*$ (see the proof of Proposition~\ref{prop:A_a-into-A(a)} for the notation). In view of $\#a,\#\sigma<\#a\star\sigma$ the induction hypothesis yields $o(A^*(\sigma))\prec o(A^*)=\omega_2^{o(A)}$ and $o(A(a))\prec o(A)$. The latter implies $o(A(a)^*)=\omega_2^{o(A(a))}\prec\omega_2^{o(A)}$. Since we are concerned with ordinals below~$\varepsilon_0$, we also have $o(A)\prec\omega_2^{o(A)}$. Using the fact that $\omega_2^{o(A)}$ is additively and multiplicatively indecomposable, we can deduce
\begin{multline*}
o(A^*(a\star\sigma))=o(A(a)^*+A(a)^*\otimes A\otimes A^*(\sigma))=\\
=o(A(a)^*)\oplus o(A(a)^*)\otimes o(A)\otimes o(A^*(\sigma))\prec\omega_2^{o(A)}=o(A^*),
\end{multline*}
as required.
\end{proof}

Recall that the terms of type~$\mathfrak B$ coincide with the finite binary trees, i.\,e.~with the element of~$\mathcal B$. Below we will show that the type $\mathfrak B[\sigma]$ is low for any non-empty bad sequence $\sigma\in\operatorname{Bad}(\mathcal B)=\operatorname{Bad}(\mathfrak B)$. To state the following definition, we simply assume that the primitive recursive function~$o(\cdot)$ is extended to arbitrary~arguments. 

\begin{definition}[$\mathbf{PA}$]
For $\sigma\in\operatorname{Bad}(\mathcal B)$ we put $r(\sigma):=o(\mathfrak B[\sigma])$.
\end{definition}

Finally, we can deduce the promised result:

\begin{corollary}[$\mathbf{PA}$]\label{cor:reification}
The primitive recursive function $r:\operatorname{Bad}(\mathcal B)\rightarrow\varepsilon_0$ is a re\-ifi\-ca\-tion, i.\,e.~we have
\begin{equation*}
r(\langle t_0,\dots,t_n,t_{n+1}\rangle)\prec r(\langle t_0,\dots,t_n\rangle)
\end{equation*}
for any bad sequence $\langle t_0,\dots,t_n,t_{n+1}\rangle$ in~$\mathcal B$.
\end{corollary}
\begin{proof}
We use induction on $\sigma\in\operatorname{Bad}(\mathfrak B)$ to show that $\mathfrak B[\sigma]$ is a low type. For this purpose it is crucial to recall that the empty sequence was included in $\operatorname{Bad}^+(\mathfrak B)$ but excluded from $\operatorname{Bad}(\mathfrak B)$. Hence the base case concerns a sequence of the form~$\sigma=\langle t\rangle$. In view of Definition~\ref{def:bad-seq-types} we have
\begin{equation*}
\mathfrak B[\langle t\rangle]=\mathfrak B[\langle\rangle](\hat e_{\mathfrak B}(\langle\rangle,t))=\mathfrak B(t).
\end{equation*}
Even though the type $\mathfrak B$ is not low, a straightforward induction on~$t\in\mathfrak B$ shows that $\mathfrak B(t)$ is a low type. Now consider a sequence~$\sigma^\frown t\in\operatorname{Bad}(\mathfrak B)$ with $\sigma\neq\langle\rangle$. The induction hypothesis ensures that~$\mathfrak B[\sigma]$ is a low type. According to Corollary~\ref{cor:bad-seq-types} we have $\hat e_{\mathfrak B}(\sigma,t)\in\mathfrak B[\sigma]$. By (the easy part of) Proposition~\ref{prop:subalgs-smaller} we conclude that
\begin{equation*}
\mathfrak B[\sigma^\frown t]=\mathfrak B[\sigma](\hat e_{\mathfrak B}(\sigma,t))
\end{equation*}
is a low type as well. The more substantial part of Proposition~\ref{prop:subalgs-smaller} yields
\begin{equation*}
r(\sigma^\frown t)=o(\mathfrak B[\sigma^\frown t])\prec o(\mathfrak B[\sigma])=r(\sigma).
\end{equation*}
For $\sigma=\langle t_0,\dots,t_n\rangle$ and $t=t_{n+1}$ this is the claim of the corollary.
\end{proof}

\section{From reflection to the finite basis property}\label{ref:refl-to-fin-base}

Working over $\mathbf{PA}$, we show that $\operatorname{Rfn}_{\mathbf{PA}}(\Sigma^0_2)$ entails $\mathcal{TI}(\varepsilon_0,\Pi^-_1)$, which does in turn entail $\mathcal K\Sigma^-_1$. This completes our proof that all three principles are equivalent. Using Goryachev's theorem, we can deduce a characterization of the $\Pi^0_1$-sentences that are provable in $\mathbf{PA}+\mathcal K\Sigma^-_1$.

For the case of uniform reflection and induction with parameters, the following has been shown by Kreisel and L\'evy~\cite{kreisel68}.

\begin{proposition}\label{prop:rfn-to-ti}
Each instance of $\mathcal{TI}(\varepsilon_0,\Pi^-_1)$ can be proved in $\mathbf{PA}+\operatorname{Rfn}_{\mathbf{PA}}(\Sigma^0_2)$.
\end{proposition}
\begin{proof}
Consider a $\Pi^0_1$-formula $\psi(x)$ with a single free variable. Arguing in the theory $\mathbf{PA}+\operatorname{Rfn}_{\mathbf{PA}}(\Sigma^0_2)$, we establish $\mathcal{TI}(\varepsilon_0,\psi)$ by contraposition: Assume that the conclusion of transfinite induction fails, so that we have $\exists_{\alpha\prec\varepsilon_0}\neg\psi(\alpha)$. The latter is a $\Sigma^0_1$-formula, so that its truth can be established by an explicit verification. More formally, we invoke formalized $\Sigma^0_1$-completeness (cf.~\cite[Theorem~I.1.8]{hajek91}) to obtain
\begin{equation*}
\exists_{\alpha\prec\varepsilon_0}\operatorname{Pr}_{\mathbf{PA}}(\neg\psi(\dot\alpha)).
\end{equation*}
This uses Feferman's dot notation: By $\psi(\dot\alpha)$ one denotes the closed object formula that result from $\psi(x)$ when we substitute $x$ by the $\alpha$-th numeral, where the code~$\alpha$ is considered as a natural number (cf.~the notation in~\cite[Corollary~I.1.76]{hajek91}). Gentzen~\cite{gentzen43} has shown that $\mathbf{PA}$ proves induction up to each fixed ordinal below~$\varepsilon_0$. This result can itself be formalized in Peano arithmetic (and in much weaker theories, cf.~\cite[Section~3]{freund-pakhomov}), so that we get
\begin{equation*}
\forall_{\alpha\prec\varepsilon_0}\operatorname{Pr}_{\mathbf{PA}}(\forall_{\gamma\prec\varepsilon_0}(\forall_{\beta\prec\gamma}\psi(\beta)\rightarrow\psi(\gamma))\to\psi(\dot\alpha)).
\end{equation*}
Together with the above this yields
\begin{equation*}
\operatorname{Pr}_{\mathbf{PA}}(\neg\forall_{\gamma\prec\varepsilon_0}(\forall_{\beta\prec\gamma}\psi(\beta)\rightarrow\psi(\gamma))).
\end{equation*}
By an instance of $\operatorname{Rfn}_{\mathbf{PA}}(\Sigma^0_2)$ we get $\neg\forall_{\gamma\prec\varepsilon_0}(\forall_{\beta\prec\gamma}\psi(\beta)\rightarrow\psi(\gamma))$, which is (provably equivalent to) a closed $\Sigma^0_2$-formula. Hence the premise of $\mathcal{TI}(\varepsilon_0,\psi)$ fails, so that our proof by contraposition is complete.
\end{proof}

The following is a consequence of the result that $(\mathcal B,\leq_{\mathcal B})$ is a well partial order with maximal order type~$\varepsilon_0$, which is due to de Jongh (unpublished; cf.~the introduction to~\cite{schmidt75}) and Diana~Schmidt (see~\cite[Theorem~II.2]{schmidt-habil} in combination with the example after~\cite[Definition~I.15]{schmidt-habil}). A detailed proof in our setting has been given in the previous section.

\begin{proposition}\label{prop:ti-to-basis}
Each instance of $\mathcal K\Sigma^-_1$ can be proved in~$\mathbf{PA}+\mathbf{TI}(\varepsilon_0,\Pi^-_1)$.
\end{proposition}
\begin{proof}
We fix an instance $\mathcal K\varphi$ (where $\varphi$ is a $\Sigma^0_1$-formula with a single free variable) and work in~$\mathbf{PA}+\mathbf{TI}(\varepsilon_0,\Pi^-_1)$. It is instructive to recall the argument from \mbox{Remark \ref{rmk:KSigma-true}}, which relies on a notion of $\varphi$-sequence. If $\{n\in\mathbb N\,|\,\varphi(n)\}$ is computably enumerable but not decidable, then it is not decidable whether a given finite sequence is a $\varphi$-sequence. For this reason we now introduce a finer notion: Write $\varphi(x)\equiv\exists_y\theta(x,y)$ with a \mbox{$\Delta^0_0$-formula}~$\theta$. As in the previous section we write $\operatorname{Bad}(\mathcal B)$ for the set of non-empty finite bad sequences in~$\mathcal B$. By a certified $\varphi$-sequence we mean a finite sequence
\begin{equation*}
(t_0,c_0),\dots,(t_n,c_n)\subseteq\mathcal B\times\mathbb N
\end{equation*}
such that we have $\langle t_0,\dots,t_n\rangle\in\operatorname{Bad}(\mathcal B)$ and $\theta(t_i,c_i)$ for all~$i\leq n$. Note that the latter implies~$\varphi(t_i)$. Since $\theta$ contains no further free variables, the notion of certified $\varphi$-sequence is defined by a $\Delta^0_1$-formula without parameters. By picking the value $f(n)$ with minimal code, one can thus define a (possibly partial) function $f:\mathbb N\to\mathcal B\times\mathbb N$ with the following property:
\begin{itemize}
\item If the sequence $\langle f(0),\dots,f(n-1)\rangle$ is defined and can be extended into a certified $\varphi$-sequence of length~$n+1$, then $\langle f(0),\dots,f(n)\rangle$ is such a sequence.
\end{itemize}
Note that the relation $f(x)=y$ is $\Sigma^0_1$-definable without parameters. Aiming at a contradiction, we now assume that the instance $\mathcal K\varphi$ is false. Then all values~$f(n)$ are defined: Inductively, we may assume that $f(m)=(t_m,c_m)$ is defined for all $m<n$; in the case of $n>0$, the construction of~$f$ ensures that $\langle(t_0,c_0),\dots,(t_{n-1},c_{n-1})\rangle$ is a certified $\varphi$-sequence. To deduce that $f(n)$ is defined as well, we consider the set $a:=\{t_0,\dots,t_{n-1}\}$. As $\mathcal K\varphi$ is false, we must have
\begin{equation*}
\neg\forall_{s\in a}\varphi(s)\lor\exists_{t\in\mathcal B}(\varphi(t)\land\forall_{s\in a}s\not\leq_{\mathcal B}t).
\end{equation*}
For $s=t_m\in a$, the construction of $f$ ensures $\theta(t_m,c_m)$ and thus $\varphi(s)$. Hence the second disjunct yields an element $t_n\in\mathcal B$ with $\varphi(t_n)$ and $t_m\not\leq_{\mathcal B}t_n$ for all~$m<n$. The latter implies $\langle t_0,\dots,t_n\rangle\in\operatorname{Bad}(\mathcal B)$. Due to $\varphi(t_n)$ we can pick a number $c_n$ with~$\theta(t_n,c_n)$. Then $\langle f(0),\dots,f(n-1),(t_n,c_n)\rangle$ is a certified $\varphi$-sequence, and $f(n)$ is defined as the smallest pair $\langle t_n,c_n\rangle$ for which this holds. We can now define a total computable function $g:\mathbb N\rightarrow\operatorname{Bad}(\mathcal B)$ by setting
\begin{equation*}
g(n):=\langle t_0,\dots,t_n\rangle\qquad\text{with $f(m)=(t_m,c_m)$}.
\end{equation*}
According to Corollary~\ref{cor:reification}, there is a primitive recursive reification
\begin{equation*}
r:\operatorname{Bad}(\mathcal B)\rightarrow\varepsilon_0.
\end{equation*}
It follows that the total computable function $r\circ g:\mathbb N\rightarrow\varepsilon_0$ is strictly decreasing. This is impossible in the presence of $\mathbf{TI}(\varepsilon_0,\Pi^-_1)$. To be more precise, we note that $r\circ g(n)=\alpha$ is $\Sigma^0_1$-definable without parameters. Using $\mathbf{TI}(\varepsilon_0,\Pi^-_1)$ one can prove
\begin{equation*}
\forall_{\alpha\prec\varepsilon_0}\forall_n\forall_{\delta\prec\varepsilon_0}(r\circ g(n)=\delta\rightarrow\alpha\preceq\delta).
\end{equation*}
To establish the induction step, it suffices to derive a contradiction from the assumption that we have $r\circ g(n)\prec\alpha$ for some~$n\in\mathbb N$. Since~$r$ is a reification, the latter would lead to $r\circ g(n+1)\prec r\circ g(n)=:\gamma$, which contradicts the induction hypothesis for~$\gamma\prec\alpha$. If we apply the result of the induction to $\alpha=r\circ g(0)+1\prec\varepsilon_0$, $n=0$ and $\delta=r\circ g(0)$, then we get $r\circ g(0)+1\preceq r\circ g(0)$, which is impossible.
\end{proof}

Together with Propositions~\ref{prop:fin-base-to-TI},~\ref{prop:TI-implies-Rfn} and~\ref{prop:rfn-to-ti} we obtain the following:

\begin{theorem}\label{thm:basis-ind-rfl}
We have
\begin{equation*}
\mathbf{PA}+\mathcal K\Sigma^-_1\,\equiv\,\mathbf{PA}+\mathcal{TI}(\varepsilon_0,\Pi^-_1)\,\equiv\,\mathbf{PA}+\operatorname{Rfn}_{\mathbf{PA}}(\Sigma^0_2),
\end{equation*}
i.\,e.~all three theories prove the same theorems.
\end{theorem}

Let $\operatorname{Con}(\mathbf{PA}+\varphi)$ be a reasonable formalization of the statement that $\mathbf{PA}+\varphi$ is consistent. We consider the recursively generated $\Pi^0_1$-sentences
\begin{align*}
\operatorname{Con}_0(\mathbf{PA})&\,:\equiv\, 0=0,\\
\operatorname{Con}_{n+1}(\mathbf{PA})&\,:\equiv\,\operatorname{Con}(\mathbf{PA}+\operatorname{Con}_n(\mathbf{PA})).
\end{align*}
Note that $\operatorname{Con}_1(\mathbf{PA})$ is equivalent to the usual consistency statement. As mentioned in the introduction, we obtain the following:

\begin{corollary}\label{cor:goryachev-basis}
We have
\begin{equation*}
\mathbf{PA}+\mathcal K\Sigma^-_1\,\equiv_{\Pi^0_1}\,\mathbf{PA}+\{\operatorname{Con}_n(\mathbf{PA})\,|\,n\in\mathbb N\},
\end{equation*}
i.\,e.~the two theories prove the same $\Pi^0_1$-sentences.
\end{corollary}
\begin{proof}
Let us write $\operatorname{Rfn}_{\mathbf{PA}}$ for the full local reflection principle, i.\,e.~the collection of all formulas $\operatorname{Pr}_{\mathbf{PA}}(\varphi)\rightarrow\varphi$, where $\varphi$ can be any sentence in the language of first order arithmetic. According to Goryachev's theorem (see e.\,g.~\cite[Theorem~IV.5]{lindstroem97}), any $\Pi^0_1$-theorem of $\mathbf{PA}+\operatorname{Rfn}_{\mathbf{PA}}$ can be proved in $\mathbf{PA}+\{\operatorname{Con}_n(\mathbf{PA})\,|\,n\in\mathbb N\}$. A~fortiori, this applies to all $\Pi^0_1$-theorems of $\mathbf{PA}+\operatorname{Rfn}_{\mathbf{PA}}(\Sigma^0_2)\equiv\mathbf{PA}+\mathcal K\Sigma^-_1$. In the other direction we have a full inclusion: The theory $\mathbf{PA}+\operatorname{Rfn}_{\mathbf{PA}}(\Sigma^0_2)$ proves all theorems of $\mathbf{PA}+\{\operatorname{Con}_n(\mathbf{PA})\,|\,n\in\mathbb N\}$, because it proves each statement~$\operatorname{Con}_n(\mathbf{PA})$. For $n=0$ this is trivial. To conclude by meta induction on~$n$, it suffices to observe that
the formula $\operatorname{Con}_n(\mathbf{PA})\rightarrow\operatorname{Con}_{n+1}(\mathbf{PA})$ is the contrapositive of
\begin{equation*}
\operatorname{Pr}_{\mathbf{PA}}(\neg\operatorname{Con}_n(\mathbf{PA}))\rightarrow\neg\operatorname{Con}_n(\mathbf{PA}),
\end{equation*}
which is an instance of~$\operatorname{Rfn}_{\mathbf{PA}}(\Sigma^0_2)$.
\end{proof}

Note that the corollary does not extend to arbitrary formula complexity: In $\mathbf{PA}+\{\operatorname{Con}_n(\mathbf{PA})\,|\,n\in\mathbb N\}$ one cannot prove all instances of~$\mathcal K\Sigma^-_1$, by Corollary~\ref{cor:Sigma2-essential}.

\nocite{ewald-sieg-hilbert}
\bibliographystyle{amsplain}
\bibliography{Independence_non-computational}

\end{document}